\documentclass[a4paper,11pt, english]{article}
\usepackage[utf8]{inputenc}
\usepackage[T1]{fontenc}
\usepackage{graphicx}
\usepackage{stmaryrd}
\usepackage{amsmath,amsfonts,amssymb,amsthm,epsfig,epstopdf,url,array}
\usepackage{rotating}
\usepackage[colorlinks=true,citecolor=red,linkcolor=blue,pdfpagetransition=Blinds]{hyperref}
\usepackage{cleveref}
\usepackage{nameref}
\usepackage{enumitem}
\usepackage{comment}
\Crefname{paragraph}{Section}{Sections}
\usepackage{fancyhdr}
\usepackage{cases}
\usepackage{mathrsfs}

\usepackage{pgfplots}

\usepackage{fullpage}

\newcommand{\ensemblenombre}[1]{\mathbb{#1}}
\newcommand{\N}{\ensemblenombre{N}}

\newcommand{\dive}[1]{\mathrm{div}}

\providecommand{\keywords}[1]{\noindent {\textit{Keywords:}} #1}

\usepackage{aliascnt}
\usepackage{cleveref}

\theoremstyle{plain}

\newtheorem{prop}{Proposition}[section]
\newaliascnt{theo}{prop}
\newtheorem{theo}[theo]{Theorem}
\aliascntresetthe{theo}

\newaliascnt{lem}{prop}
\newtheorem{lem}[lem]{Lemma}
\aliascntresetthe{lem}

\newaliascnt{defprop}{prop}

\aliascntresetthe{defprop}

\newaliascnt{cor}{prop}
\newtheorem{cor}[cor]{Corollary}
\aliascntresetthe{cor}

\newaliascnt{rmk}{prop}
\newtheorem{rmk}[rmk]{Remark}
\aliascntresetthe{rmk}

\theoremstyle{definition}

\newaliascnt{defi}{prop}
\newtheorem{defi}[defi]{Definition}
\aliascntresetthe{defi}

\newaliascnt{app}{prop}

\aliascntresetthe{app}

\newaliascnt{claim}{prop}
\newtheorem{claim}[claim]{Fact}
\aliascntresetthe{claim}

\newaliascnt{ass}{prop}
\newtheorem{ass}[ass]{Assumption}
\aliascntresetthe{ass}

\newaliascnt{example}{prop}
\newtheorem{example}[ass]{Example}
\aliascntresetthe{example}

\crefname{proposition}{Proposition}{Propositions}
\Crefname{proposition}{Proposition}{Propositions}

\crefname{theo}{Theorem}{Theorems}
\Crefname{theo}{Theorem}{Theorems}

\crefname{lemma}{Lemma}{Lemmas}
\Crefname{lemma}{Lemma}{Lemmas}

\crefname{defprop}{Definition--Proposition}{Definition--Propositions}
\Crefname{defprop}{Definition--Proposition}{Definition--Propositions}

\crefname{cor}{Corollary}{Corollaries}
\Crefname{cor}{Corollary}{Corollaries}

\crefname{rmk}{Remark}{Remarks}
\Crefname{rmk}{Remark}{Remarks}

\crefname{defi}{Definition}{Definitions}
\Crefname{defi}{Definition}{Definitions}

\crefname{app}{Application}{Applications}
\Crefname{app}{Application}{Applications}

\crefname{claim}{Claim}{Claims}
\Crefname{claim}{Claim}{Claims}

\crefname{ass}{Assumption}{Assumptions}
\Crefname{ass}{Assumption}{Assumptions}

\def\dt{\textnormal{d}t}
\def\d{\textnormal{d}}
\def\esp{{\mathbb{E}}}
\def\dom{\mathcal{D}}
\def\fil{\mathbb{F}}
\newcommand{\F}{\mathcal{F}}
\newcommand{\vertiii}[1]{{\left\vert\kern-0.25ex\left\vert\kern-0.25ex\left\vert #1 
    \right\vert\kern-0.25ex\right\vert\kern-0.25ex\right\vert}}

\makeatletter
\let\original@addcontentsline\addcontentsline
\newcommand{\dummy@addcontentsline}[3]{}
\newcommand{\DeactivateToc}{\let\addcontentsline\dummy@addcontentsline}
\newcommand{\ActivateToc}{\let\addcontentsline\original@addcontentsline}
\makeatother

\begin{document}

\title{Stability for the stochastic heat equation with multiplicative noise via finite-dimensional feedback}
\author{V\'ictor Hern\'andez-Santamar\'ia\thanks{V. Hern\'andez-Santamar\'ia is supported by Project CBF2023-2024-116 of SECIHTHI and by UNAM-DGAPA-PAPIIT grants IA103826, IN117525, and IN102925 (Mexico).} \and  K\'evin Le Balc'h\thanks{K. Le Balc'h received support from Project ANR-20-CE40-0009: TRECOS.} \and Liliana Peralta\thanks{L. Peralta has received support from UNAM-DGAPA-PAPIIT IA103826.}}

\maketitle

\begin{abstract}
In this paper, we study the long-time behavior of a stochastic
heat equation with {a class of infinite-dimensional
multiplicative noises} and localized control. {We first illustrate through simple examples of the uncontrolled dynamics that mean-square and almost sure exponential stability may hold under different conditions on the parameters, reflecting the interplay between the drift and the multiplicative noise.} We then construct a finite-dimensional feedback control acting on a measurable subset of positive measure, built from finitely many Fourier modes of the solution. In particular, we show that the number of controlled modes determines the decay rate and allows for arbitrarily fast stabilization in the mean-square sense. As a consequence, almost sure exponential stability is recovered via a probabilistic argument, so that both notions of stability are achieved within the same framework and {with the same guaranteed decay rate}. {We also establish null-controllability at any positive time through an iterative construction of adapted controls from finite-dimensional feedback laws. The construction relies on the quantitative stabilization estimates and requires no observability inequality for the adjoint equation.}
\end{abstract}

\keywords{Finite-dimensional control; rapid stabilization; spectral methods; almost sure convergence} 

\smallskip
\noindent
\textit{2020 MSC:} {60H15, 35B35, 93D15, 93B05}

\footnotesize
\tableofcontents
\normalsize

\section{Introduction}\label{sec:intro}

\subsection{Motivation}\label{sec:motiv}
The analysis of the long-time behavior of stochastic parabolic equations, and in particular the relationship between different probabilistic notions of stability, is a central topic in the theory of stochastic partial differential equations. To provide a concrete reference point, we first consider the following deterministic parabolic system
\begin{equation}\label{eq:det_parab}
\begin{cases}
y_t=\Delta y + c_0 y &\textnormal{in } (0,T)\times \dom, \\
y=0 & \textnormal{in } (0,T)\times \Gamma, \\
y(0)=y_0 &\textnormal{in } \dom,
\end{cases}
\end{equation}
where $T>0$, $\dom$ is a bounded, connected, open subset of $\mathbb R^d$ ($d\geq 1$) with smooth boundary $\Gamma=\partial \dom$, and $c_0\in\mathbb R$. By classical arguments, for every $y_0 \in L^2(\dom)$, system \eqref{eq:det_parab} is globally well-posed and admits a unique weak solution $y \in C([0,T];L^2(\dom))$. Moreover, this solution satisfies
\begin{equation}\label{eq:energy}
\|y(t)\|^2_{L^2(\dom)}\leq e^{-2 (\tau_1-c_0) t}\|y_0\|_{L^2(\dom)}^2,
\quad \textnormal{for all } t\geq 0,
\end{equation}
where $\tau_1$ denotes the first eigenvalue of the Laplacian with homogeneous Dirichlet boundary conditions. In particular, if $c_0<\tau_1$, the solution decays exponentially as $t\to+\infty$, whereas if $c_0>\tau_1$, certain solutions may grow exponentially fast.

In stochastic modeling, uncertainties and random fluctuations are naturally incorporated through random perturbations of the coefficients. As a motivating example, we consider a scalar multiplicative perturbation, formally replacing the potential $c_0$ by
\begin{equation*}
c + a \dot{W}(t),
\end{equation*}
where $c \in \mathbb{R}$, $a \in \mathbb{R} \setminus \{0\}$, and $\dot{W}(t)$ denotes the (formal) time derivative of a standard one-dimensional Brownian motion $W(t)$. We work on a complete filtered probability space $(\Omega, \mathcal{F}, \fil, \mathbb{P})$, where $\fil = \{\mathcal F_t\}_{t \ge 0}$ satisfies the usual conditions and $W$ is a Brownian motion with respect to $\fil$.

We are thus led to consider the stochastic partial differential equation, understood in the It\^o sense,
\begin{equation}\label{eq:spde}
\begin{cases}
\d y = (\Delta y + c y)\dt + a y \d W(t) & \textnormal{in } (0,T)\times \dom, \\
y = 0 & \textnormal{on } (0,T)\times \Gamma, \\
y(0) = y_0 & \textnormal{in } \dom,
\end{cases}
\end{equation}
which, for any $T>0$ and initial condition $y_0 \in L^2_{\mathcal F_0}(\Omega;L^2(\dom))$, admits a unique weak solution (see \Cref{def:weak_solution}).

This raises the question of how the stability properties of \eqref{eq:det_parab} are affected in the stochastic setting. In this context, stability can be interpreted in several ways. Two commonly used notions are mean-square exponential stability, which concerns the decay of the expected squared norm of the solution, and almost sure exponential stability, which describes the long-time behavior of individual sample paths.

\begin{defi}[Mean-square exponential stability]
The solution $y$ of \eqref{eq:spde} is said to be {mean-square exponentially stable} if there exist constants $C>0$ and $\mu>0$ such that
\begin{equation*}
\mathbb{E}\big(\|y(t)\|_{L^2(\dom)}^2\big)
\le C\, e^{-\mu t} \,
\mathbb{E}\big(\|y_0\|_{L^2(\dom)}^2\big),
\quad \text{for all } t \ge 0.
\end{equation*}
\end{defi}

\begin{defi}[Almost sure exponential stability]\label{ases}
The solution $y$ of \eqref{eq:spde} is said to be {almost surely exponentially stable} if there exists $\mu>0$ such that
\begin{equation*}
\limsup_{t \to \infty} \frac{1}{t} \log \|y(t)\|_{L^2(\dom)}^2
\le -\mu
\quad \text{a.s.}
\end{equation*}
\end{defi}

The following classical facts describe the long-time behavior of solutions to \eqref{eq:spde} and highlight that mean-square and almost sure stability may hold under different conditions on the parameters $c$, $a$, and $\tau_1$. For completeness, we include short proofs in \Cref{app_stab}. Both estimates follow directly from the explicit representation of the solution to \eqref{eq:spde} given by
$y(t)=\exp((c-a^2/2)t+aW(t))e^{t\Delta}y_0$,
where $e^{t\Delta}$ denotes the Dirichlet heat semigroup.
\begin{claim}
If $2\tau_1 > a^2 + 2c$, then the solution to \eqref{eq:spde} is mean-square exponentially stable with decay rate $\mu = 2(\tau_1 - c) - a^2$.
\end{claim}
\begin{claim}
If $2\tau_1 + a^2 > 2c$, then the solution is almost surely exponentially stable with decay rate $\mu = 2(\tau_1 - c) + a^2$.
\end{claim}

In particular, multiplicative noise may either destroy mean-square stability or induce almost sure exponential stabilization, depending on its intensity and on the sign of the drift. The latter phenomenon, often referred to as \emph{stabilization by noise}, has been widely studied in the literature; see, e.g., \cite{Arn79,Has80,Arn90,Mao94,Hau78,CLM01}.

These observations reveal a fundamental feature of the stochastic dynamics: unlike in the deterministic case, different notions of stability are no longer equivalent and need not hold under the same conditions. Moreover, although mean-square exponential stability implies almost sure exponential stability in the present setting (see, e.g., \cite[Theorem~4.2]{Mao08}), such results do not improve the stability regime, as they require the same conditions under which mean-square stability holds.

{
The main objective of this work is to establish quantitative
stabilization results for the stochastic heat equation with
a broader class of multiplicative noises, introduced below.
We construct a finite-dimensional feedback acting on a prescribed
measurable subset of positive measure and obtain mean-square
exponential stability with an arbitrarily large decay rate,
together with quantitative bounds on the control.
As applications of these estimates, we derive almost sure
exponential stability with at least the prescribed decay rate
and null-controllability at any positive time through an
iterative construction of adapted controls in feedback form.
}

\subsection{Main results}

Motivated by the discussion above, we introduce a controlled version of the stochastic heat equation, with a more general multiplicative noise, where the control acts on a subregion of the domain. More precisely, let $\dom_0 \subset \dom$ be a measurable subset with positive Lebesgue measure, that is, $|\dom_0|>0$, and consider the system
\begin{equation}\label{eq:forward_semilinear_intro}
\begin{cases}
\d y = (\Delta y + c y + \chi_{\dom_0} h)\,\dt
+ \displaystyle\sum_{\ell\geq 1}\sqrt{q_\ell}\,G_\ell y\,\d\beta_\ell(t)
& \text{in } (0,T)\times\dom,\\
y = 0 & \text{on } (0,T)\times \Gamma,\\
y(0) = y_0 & \text{in } \dom,
\end{cases}
\end{equation}
where $h$ is a forcing term that will be chosen as a control function depending on the solution itself. Note that the action of the control $h$ is prescribed in the subregion $\mathcal{D}_0$. The coefficients and operators defining the noise are specified below.

To this end, consider the differential operator $\Delta : H^2(\dom) \cap H_0^1(\dom) \longrightarrow L^2(\dom)$. Then $-\Delta$ is a positive self-adjoint operator with compact resolvent. As a consequence, there exists an orthonormal basis of $L^2(\dom)$ consisting of eigenfunctions of $-\Delta$. More precisely, there exist sequences $\{e_k\}_{k\in\mathbb{N}^*} \subset H^2(\dom)\cap H_0^1(\dom)$ and $\{\tau_k\}_{k\in\mathbb{N}^*} \subset (0,+\infty)$ such that, for every $k\in\mathbb{N}^*$,
\begin{equation*}
\begin{cases}
-\Delta e_k = \tau_k e_k & \text{in } \dom,\\
e_k = 0 & \text{on } \partial \dom,
\end{cases}
\end{equation*}
and
\begin{equation*}
(e_k,e_\ell)_{L^2(\dom)} = \delta_{k\ell},
\qquad
w = \sum_{k=1}^{+\infty} (w,e_k)_{L^2(\dom)} e_k,
\quad \forall w \in L^2(\dom).
\end{equation*}
Moreover, the eigenvalues satisfy
\begin{equation*}
0 < \tau_1 < \tau_2 \le \tau_3 \le \cdots,
\qquad
\tau_k \to +\infty \quad \text{as } k \to +\infty.
\end{equation*}
Each eigenvalue $\tau_k$ has finite multiplicity. In addition, since $\dom$ is connected, the first eigenvalue $\tau_1$ is simple.

{
We now specify the noise in \eqref{eq:forward_semilinear_intro}.
{
Let $\{\beta_\ell\}_{\ell\geq1}$ be independent standard Brownian
motions with respect to $\fil$, and let $\{q_\ell\}_{\ell\geq1}$
be positive numbers satisfying
\begin{equation*}
\sum_{\ell\geq1}q_\ell<\infty.
\end{equation*}
}

For each $\ell\geq1$, let
$G_\ell:L^2(\dom)\to L^2(\dom)$ be a bounded linear operator,
diagonal with respect to the eigenbasis of the Dirichlet Laplacian:
\begin{equation}\label{eq:diagonal_noise}
G_\ell e_k=g_{\ell k}e_k,
\qquad g_{\ell k}\in\mathbb R,
\qquad \ell,k\geq1.
\end{equation}
We assume that
\begin{equation}\label{eq:noise_bound}
C_G:=
\sup_{k\geq1}
\sum_{\ell\geq1}q_\ell|g_{\ell k}|^2
<\infty.
\end{equation}

For every $v\in L^2(\dom)$, using
\eqref{eq:diagonal_noise}, \eqref{eq:noise_bound}, and Tonelli's theorem, we obtain
\begin{align}\label{eq:est_ruido}
\sum_{\ell\geq1}q_\ell\|G_\ell v\|_{L^2(\dom)}^2
&=
\sum_{k\geq1}
\left(
\sum_{\ell\geq1}q_\ell|g_{\ell k}|^2
\right)
|(v,e_k)_{L^2(\dom)}|^2\\
&\leq C_G\|v\|_{L^2(\dom)}^2.
\end{align}
Consequently, if $y$ is progressively measurable and satisfies
$\mathbb E\int_0^T\|y(t)\|_{L^2(\dom)}^2\,\dt<\infty$, then
\begin{equation*}
\mathbb E\int_0^T
\sum_{\ell\geq1}q_\ell
\|G_\ell y(t)\|_{L^2(\dom)}^2\,\dt
\leq
C_G\mathbb E\int_0^T
\|y(t)\|_{L^2(\dom)}^2\,\dt
<\infty.
\end{equation*}
Therefore, the series
\begin{equation*}
\sum_{\ell\geq1}\sqrt{q_\ell}
\int_0^t G_\ell y(s)\,\d\beta_\ell(s)
\end{equation*}
converges in
$L^2(\Omega;C([0,T];L^2(\dom)))$, see \cite[Chapter 2]{LR15}.
This defines the stochastic term in
\eqref{eq:forward_semilinear_intro}.
{
This construction fits the framework of integration with respect
to a $Q$-Wiener process in \cite[Chapter 2]{LR15}, with covariance
eigenvalues $q_\ell$. We use its series representation throughout
the paper.
}

\begin{example}
The above framework includes, in particular, the following cases.

\begin{enumerate}

\item[(i)] The scalar multiplicative noise considered in the motivation
is recovered by taking
\begin{equation*}
q_1=1,
\qquad
G_1=aI,
\qquad
G_\ell=0,\quad \ell\geq2.
\end{equation*}

\item[(ii)] Let $\Pi_\ell$ be the orthogonal projection onto
$\operatorname{span}\{e_\ell\}$, and let
$\{\sigma_\ell\}_{\ell\geq1}\subset\mathbb R$ satisfy
\begin{equation*}
\sup_{\ell\geq1}q_\ell|\sigma_\ell|^2<\infty.
\end{equation*}
Setting
\begin{equation*}
G_\ell=\sigma_\ell\Pi_\ell,
\qquad \ell\geq1,
\end{equation*}
we have
\begin{equation*}
g_{\ell k}=\sigma_\ell\delta_{\ell k},
\end{equation*}
where $\delta_{\ell k}$ denotes the Kronecker delta. Thus, each
Brownian motion acts only on the corresponding spectral mode, and
\eqref{eq:noise_bound} is satisfied.

\item[(iii)] Let $\{\sigma_k\}_{k\geq1}\subset\mathbb R$ be a bounded sequence, and
take
\begin{equation*}
q_1=1,
\qquad
G_1e_k=\sigma_ke_k,
\qquad
G_\ell=0,\quad \ell\geq2.
\end{equation*}
Then a single Brownian motion acts on all spectral modes with possibly
different intensities, and
\begin{equation*}
C_G=\sup_{k\geq1}|\sigma_k|^2<\infty.
\end{equation*}
The scalar noise in {\rm(i)} is recovered when $\sigma_k=a$ for every
$k\geq1$.

\end{enumerate}
\end{example}
}

\subsubsection{Stabilization results}

The discussion in the introduction suggests that stability cannot always be guaranteed by the balance between the drift and the noise. We therefore introduce a finite-dimensional feedback acting on the low frequencies of the solution.

More precisely, given a parameter $\lambda>0$, we define
\begin{equation}\label{N_lambda}
N_\lambda := \#\{ i \in \mathbb{N}^* : \tau_i \le \lambda \},
\end{equation}
that is, the number of eigenvalues of $-\Delta$ (counted with multiplicity) that do not exceed $\lambda$.
We also denote by $P_{N_\lambda} w := \sum_{i=1}^{N_\lambda} (w,e_i)_{L^2(\dom)}\, e_i$, $w \in L^2(\dom)$, the orthogonal projection onto the span of the first $N_\lambda$ eigenfunctions.

{
Note that by \eqref{eq:diagonal_noise}, the noise operators $G_\ell$ satisfy
\begin{equation*}
P_{N_\lambda}G_\ell=G_\ell P_{N_\lambda},
\qquad \ell\geq1.
\end{equation*}
}


With these elements at hand, we define the feedback operator
\begin{equation}\label{eq:def_feedback_intro}
\mathscr H_\lambda y := - \gamma_\lambda P_{N_\lambda} y
= - \gamma_\lambda \sum_{i=1}^{N_\lambda} (y,e_i)_{L^2(\dom)}\, e_i,
\end{equation}
where $\gamma_\lambda>0$ is a parameter to be chosen depending on $\lambda$.

This choice allows us to rewrite \eqref{eq:forward_semilinear_intro} as the closed-loop system
\begin{equation}\label{eq:forward_semilinear_feed_intro}
\begin{cases}
\d y = (\Delta y + c y + \chi_{\dom_0} \mathscr H_\lambda y)\,\dt
+ {\displaystyle\sum_{\ell\geq1}\sqrt{q_\ell}\,G_\ell y\,\d\beta_\ell(t)}
& \text{in } {(0,+\infty)}\times \dom,\\
y = 0 & \text{on } {(0,+\infty)}\times \Gamma,\\
y(0) = y_0 & \text{in } \dom.
\end{cases}
\end{equation}

Our first result shows that this feedback law ensures mean-square exponential stability with a prescribed decay rate $\lambda$, as soon as $\lambda$ is chosen sufficiently large.

\begin{theo}[Mean-square stabilization by finite-dimensional feedback]\label{theo:main}
{Assume \eqref{eq:diagonal_noise}--\eqref{eq:noise_bound}.}
Let $\tau_1$ be the first eigenvalue of the Dirichlet Laplacian and fix
$\lambda > \max\{2\tau_1,\, {C_G}+2c+{\tau_1}\}$.
{There exists a choice of $\gamma_\lambda>0$ such that, for every
$y_0\in L^2_{\mathcal F_0}(\Omega;L^2(\dom))$,}
the closed-loop system \eqref{eq:forward_semilinear_feed_intro} admits a unique weak solution
\begin{equation*}
y \in L^2_{\fil}(\Omega;C([0,\infty);L^2(\dom))) \cap L^2_{\fil}(0,\infty;H_0^1(\dom)),
\end{equation*}
in the sense of \Cref{def:weak_solution} (see \Cref{sec:functional}). Moreover, the solution is exponentially stable in mean square with decay rate $\lambda$. More precisely, there exists a constant $C>0$, depending only on $\dom$, $\dom_0$, {$C_G$}, and $c$, such that
\begin{equation}\label{eq:decay_unified}
\esp\!\left(\|y(t)\|^2_{L^2(\dom)}\right)
\le C\, e^{C\sqrt{\lambda}} e^{-\lambda t}
\esp\!\left(\|y_0\|^2_{L^2(\dom)}\right),
\qquad \forall t\ge 0.
\end{equation}

In addition, the corresponding feedback control satisfies
\begin{equation}\label{eq:bound_ctrl_unified}
\esp\!\left(\|\mathscr H_{\lambda}y(t)\|^2_{L^2(\dom)}\right)
\le C\, e^{C\sqrt{\lambda}} e^{-\lambda t}
\esp\!\left(\|y_0\|^2_{L^2(\dom)}\right),
\qquad \forall t\ge 0.
\end{equation}
\end{theo}

The proof is based on a decomposition of the solution into low and high frequencies with respect to the Dirichlet Laplacian. {A spectral inequality on measurable sets (see \Cref{prop:spectral}) allows the localized feedback to stabilize the low-frequency component.} The high-frequency component is controlled by the dissipation of the equation. Suitable weighted energy estimates are derived for both parts and then combined through an appropriate choice of the weight and of the feedback gain $\gamma_\lambda$.

This argument extends to the stochastic setting the stabilization strategy developed in \cite{Xia24}\footnote{This approach, based on finite-dimensional feedback acting on low frequencies, has recently attracted attention and has been extended and adapted in several works, often with different objectives; see, e.g., \cite{Xia23,LWXZZ24,CX25,CGP25,AHSM26}.} in the deterministic case.
In the stochastic literature, related results were obtained in \cite{BLT02,Bar13}, where the feedback depends on the full solution, localized in space. In both cases, however, restrictive geometric assumptions are imposed on the control region, requiring in particular that $\dom\setminus\dom_0$ be sufficiently {thin}. Moreover, these approaches provide exponential stabilization under suitable conditions, but {do not provide arbitrarily fast stabilization on an arbitrary prescribed control region}.

{
\begin{rmk}
The commutation property
\begin{equation*}
P_{N_\lambda}G_\ell=G_\ell P_{N_\lambda},
\qquad \ell\geq1,
\end{equation*}
which follows from \eqref{eq:diagonal_noise}, plays an essential role in the
proof of \Cref{theo:main}; see \Cref{sec:mean_s}. Indeed, it prevents the
multiplicative noise from coupling the low- and high-frequency components.
As a consequence, the low-frequency coefficients satisfy a closed
finite-dimensional system, whereas the high-frequency equation is forced
by the localized feedback through the low-frequency component.

However, it is worth noting that exact commutation may be stronger than
necessary. In deterministic parabolic problems, noncommuting spectral
decompositions can still be handled when the coupling between spectral
components is compensated by a sufficiently strong quantitative
high-frequency dissipation estimate, see, for instance, \cite{BPS18}. This suggests that the present compatibility assumption could potentially be relaxed by replacing exact invariance of the low/high-frequency decomposition with suitable quantitative estimates on the interaction between the corresponding spectral components. This observation leads us to consider the extension to more
general multiplicative noises as an interesting problem
for future work.
\end{rmk}
}

{
\begin{rmk}
During the revision of this manuscript, the related work
\cite{AvHR26} appeared, addressing the stabilization of nonlinear
nonautonomous stochastic parabolic equations driven by a
$Q$-Wiener process, without requiring the spectral compatibility
assumed here. The approach in \cite{AvHR26} uses oblique
projections and finitely many indicator-type actuators, whose
supports are chosen as part of the feedback construction.
In the present work, the control region is any prescribed
measurable subset of positive measure, and the feedback is built
from the low-frequency modes of the Dirichlet Laplacian.
Under the spectral assumptions on the multiplicative noise,
we obtain quantitative estimates in terms of the prescribed
decay rate. These estimates also yield null-controllability
through an iterative construction of adapted controls in
feedback form; see \Cref{thm:null_feed} below.
\end{rmk}
}

{
The next result shows that the same feedback law also ensures
almost sure exponential decay with at least the prescribed
rate $\lambda$.
}

\begin{theo}\label{thm:as_intro}
Under the assumptions of Theorem~\ref{theo:main}, {and with the same choice of feedback,} the solution to the closed-loop system \eqref{eq:forward_semilinear_feed_intro} is almost surely exponentially stable with decay rate {at least} $\lambda$. More precisely,
\begin{equation*}
\limsup_{t\to\infty} \frac{1}{t}\log \|y(t)\|^2_{L^2(\dom)} \leq -\lambda
\quad \text{a.s.}
\end{equation*}
\end{theo}

{
The proof combines estimates for the expected supremum of the
solution with a Borel--Cantelli argument on a sequence of time
intervals.
}
This strategy is based on classical arguments in the theory of stochastic evolution equations, which can be traced back to \cite{Hau78,Ich82} and have been further developed in many works; see in particular \cite{Car90,Car03,Car02}. {Here, the quantitative estimates in \Cref{theo:main} allow us to retain the prescribed decay rate $\lambda$ at the almost sure level.}

\subsubsection{Null-controllability by feedback}

We conclude this section by showing that the
{quantitative stabilization estimates yield
null-controllability at any positive time, with controls
constructed from finite-dimensional feedback laws.}

Given a time horizon $T>0$, we consider the controlled system
\begin{equation}\label{eq:stoch-heat-control}
\begin{cases}
\d y = (\Delta y + c y + \chi_{\dom_0} h)\,\dt
+ \displaystyle\sum_{\ell\geq1}\sqrt{q_\ell}\,G_\ell y\,\d\beta_\ell(t)
& \text{in } (0,T)\times \dom,\\
y = 0 & \text{on } (0,T)\times \Gamma,\\
y(0) = y_0 & \text{in } \dom,
\end{cases}
\end{equation}
and we say that the system is null-controllable at time $T$ if for every initial datum $y_0 \in L^2_{\mathcal F_0}(\Omega;L^2(\dom))$ there exists a control $h \in L^2_{\fil}(0,T;L^2(\dom_0))$ such that the corresponding solution satisfies $y(T)=0 \quad \text{a.s.}$.

The null-controllability of stochastic parabolic equations has been extensively studied in the linear setting, see, for instance, \cite{BRT03,TZ09,Lu11,Liu14}, and more recently \cite{HSLBP20a,HSLBP23} for extensions to nonlinear problems. A standard approach is based on the Lebeau--Robbiano strategy, which combines spectral inequalities with an iterative procedure on short time intervals in order to drive the solution to zero. In the stochastic setting, this strategy was implemented in \cite{Lu11}, for stochastic heat equations with control acting only on the drift.

Our result reads as follows.

\begin{theo}[Null-controllability by feedback]\label{thm:null_feed}
Assume \eqref{eq:diagonal_noise}--\eqref{eq:noise_bound}.
For every $T>0$ and every
$y_0 \in L^2_{\mathcal F_0}(\Omega;L^2(\dom))$,
there exists an $\fil$-adapted control $h \in L^2_{\fil}(0,T;L^2(\dom_0))$,
constructed from a family of finite-dimensional feedback laws, such that the corresponding solution to \eqref{eq:stoch-heat-control} satisfies
\begin{equation*}
y(T)=0 \quad \text{in } L^2(\dom), \ \text{a.s.}
\end{equation*}
\end{theo}

{
\Cref{thm:null_feed} provides a constructive null-controllability
result on any prescribed measurable control region of positive
measure. The result allows infinitely many independent Brownian
motions acting with different intensities on the spectral modes,
under assumptions \eqref{eq:diagonal_noise}--\eqref{eq:noise_bound}.
}
More precisely, by combining
{the quantitative bounds for the solution and the
feedback control in \Cref{theo:main}}
with an iterative procedure on shrinking time intervals, we construct controls directly in feedback form. This can be interpreted as a feedback version of the Lebeau--Robbiano strategy. The proof is inspired by \cite[Theorem 3.1]{Xia24}, which we adapt to the present stochastic setting.

A key feature of this approach is that the control is constructed directly, without passing through the adjoint equation. In the stochastic setting, this adjoint corresponds to a backward SPDE, and the derivation of observability inequalities for such equations is a delicate issue; see \cite[Section 5]{Bar13}. The present construction does not require such an observability inequality and yields adapted controls by construction.

More generally, this suggests that whenever a suitable spectral inequality is available and the noise is compatible with the spectral decomposition, analogous quantitative stabilization and control estimates may allow the construction of adapted controls through an iterative feedback procedure. This provides a possible direction for systems whose adjoint equations are difficult to analyze, such as fluid models (e.g., the stochastic Stokes system), for which spectral inequalities are available (see \cite{CSL16,CSFSS25}). In this sense, feedback-based strategies provide a viable alternative to classical approaches.

\subsection{Organization of the paper}

The paper is organized as follows. In Section~\ref{sec:functional}, we introduce the functional framework and recall well-posedness results for the stochastic heat equation. In Section~\ref{sec:mean_s}, we establish the mean-square stabilization result. Section~\ref{sec:as_s} is devoted to the proof of almost sure exponential stability. Finally, in Section~\ref{sec:nc}, we prove the null-controllability result based on the feedback construction.

\section{Functional setting, well-posedness and spectral tools}\label{sec:functional}

In this section, we make precise the functional framework for the stochastic systems introduced in \Cref{sec:intro} and recall the corresponding well-posedness results. We also present a spectral inequality that will play a central role in the analysis.

Let $T>0$ be fixed and let $\dom\subset\mathbb R^d$ be the domain introduced in the introduction. We denote by $Q_T=(0,T)\times\dom$ and $\Sigma_T=(0,T)\times\partial\dom$.

In what follows, we denote $\{\F_t\}_{t\geq 0}$ by $\fil$. Let $(X,\|\cdot\|_X)$ be a Banach space. We denote by $L^2_{\mathcal F_s}(\Omega;X)$, $s\geq0$, the space of all $\mathcal F_s$-measurable $X$-valued random variables $\xi$ such that $\mathbb E(\|\xi\|_X^2)<+\infty$. We also denote by $L^2_{\fil}(0,T;X)$ the space of all $X$-valued progressively measurable processes $\psi$ such that $\mathbb E\!\left(\int_0^T\|\psi(t)\|_X^2\,dt\right)<+\infty$, and by $L^2_{\fil}(\Omega;C([0,T];X))$ the space of all continuous $\fil$-adapted processes $\psi$ such that $\mathbb E\!\left(\|\psi\|_{C([0,T];X)}^2\right)<+\infty$. We use the analogous notation on any (deterministic) interval $[s,T]$.

We consider stochastic parabolic equations of the form
\begin{equation}\label{eq:general_forward}
\begin{cases}
\d y = \left(\Delta y + cy + F(y)+f\right)\dt
+ \displaystyle\sum_{\ell\geq1}\sqrt{q_\ell}\,G_\ell y\,\d\beta_\ell(t)
& \textnormal{in } Q_T,\\
y = 0 & \textnormal{on } \Sigma_T,\\
y(0) = y_0 & \textnormal{in } \dom,
\end{cases}
\end{equation}
where $c\in\mathbb R$,
$y_0 \in L^2_{\mathcal F_0}(\Omega;L^2(\dom))$,
$f\in L^2_{\fil}(0,T;L^2(\dom))$, and
$F:L^2(\dom)\to L^2(\dom)$ is a (possibly nonlinear) mapping.
The noise satisfies \eqref{eq:diagonal_noise}--\eqref{eq:noise_bound}.

\begin{defi}\label{def:weak_solution}
A process $y$ is called a weak solution to \eqref{eq:general_forward} if
\begin{enumerate}
\item[1)] $y$ is $L^2(\dom)$-valued and $\fil$-adapted,
\item[2)] $y\in L^2_{\fil}(\Omega;C([0,T];L^2(\dom)))\cap L^2_{\fil}(0,T;H_0^1(\dom))$,
\item[3)] for any $\phi\in H_0^1(\dom)$, it holds
\begin{align*}
(y(t),\phi)_{L^2(\dom)}
&= (y_0,\phi)_{L^2(\dom)}
- \int_0^t (\nabla y,\nabla \phi)_{L^2(\dom)}\d{s}  + \int_0^t (cy+F(y)+f,\phi)_{L^2(\dom)}\d{s}\\
&\quad + \sum_{\ell\geq1}\sqrt{q_\ell}
\int_0^t (G_\ell y,\phi)_{L^2(\dom)}\,\d\beta_\ell(s),
\qquad \textnormal{a.s. for all }t\in[0,T].
\end{align*}
\end{enumerate}
\end{defi}

We assume that $F$ is globally Lipschitz, that is, there exists $L>0$ such that
$\|F(y_1)-F(y_2)\|_{L^2(\dom)}
\leq L\|y_1-y_2\|_{L^2(\dom)}$
for all $y_1,y_2\in L^2(\dom)$.
We also assume that $F(0)=0$.

\begin{theo}\label{thm:general_wellposed}
For any $y_0\in L^2_{\mathcal F_0}(\Omega;L^2(\dom))$
and $f\in L^2_{\fil}(0,T;L^2(\dom))$,
system \eqref{eq:general_forward} admits a unique weak solution
in the sense of Definition \ref{def:weak_solution}.

Moreover, there exists a constant $C>0$, depending only on
$T$, $\dom$, $c$, $L$, and $C_G$, such that
\begin{align}\label{eq:wellposed_estimate}
&\mathbb E\left(\sup_{t\in[0,T]}
\|y(t)\|_{L^2(\dom)}^2\right)
+\mathbb E\int_0^T\|\nabla y(t)\|_{L^2(\dom)}^2\,\dt
\leq C\left(
\mathbb E\|y_0\|_{L^2(\dom)}^2
+\mathbb E\int_0^T\|f(t)\|_{L^2(\dom)}^2\,\dt
\right).
\end{align}

\end{theo}

The result follows from the standard well-posedness theory for stochastic parabolic equations with globally Lipschitz coefficients; see, e.g., \cite{LR15}. Indeed, the noise estimate established in \eqref{eq:est_ruido} gives
\begin{equation*}
\sum_{\ell\geq1}q_\ell
\|G_\ell v-G_\ell w\|_{L^2(\dom)}^2
\leq C_G\|v-w\|_{L^2(\dom)}^2,
\qquad v,w\in L^2(\dom).
\end{equation*}
Together with the Lipschitz condition on $F$, this yields existence and uniqueness by Galerkin approximation. Estimate \eqref{eq:wellposed_estimate} follows from It\^o's formula, the Burkholder--Davis--Gundy inequality, and Gronwall's lemma. For brevity, we omit the details.

Note that the same result stated in \Cref{thm:general_wellposed} holds on any (deterministic) interval $[s,T]$, with initial datum in $L^2_{\mathcal F_s}(\Omega;L^2(\dom))$ and forcing in $L^2_{\fil}(s,T;L^2(\dom))$. The constant in the corresponding estimate depends on $T-s$ in place of $T$.

As a particular case, we consider the controlled system
\begin{equation}\label{eq:forward_gen}
\begin{cases}
\d y=\left(\Delta y + cy +\chi_{\dom_0}\left[\sum_{i=1}^N\ell_i(y)\varphi_i\right]\right)\dt
+\displaystyle\sum_{\ell\geq1}\sqrt{q_\ell}\,G_\ell y\,\d\beta_\ell(t)
&\text{in }Q_T,\\
y=0 &\text{on }\Sigma_T,\\
y(0)=y_0 &\text{in }\dom,
\end{cases}
\end{equation}
where $N\in\mathbb N$, $\varphi_i\in L^2(\dom)$,
$\ell_i:L^2(\dom)\to\mathbb R$ are bounded linear operators,
and $\dom_0\subset\dom$ is a measurable subset with positive Lebesgue measure $|\dom_0|>0$.
This corresponds to the choice
$F(y)=\chi_{\dom_0}\sum_{i=1}^N \ell_i(y)\varphi_i$
and $f=0$.

Since the operators $\ell_i$ are bounded and $\varphi_i\in L^2(\dom)$, the mapping $F$ is globally Lipschitz in $L^2(\dom)$. Therefore, by Theorem \ref{thm:general_wellposed}, system \eqref{eq:forward_gen} admits a unique weak solution for every $y_0\in L^2_{\mathcal F_0}(\Omega;L^2(\dom))$. Likewise, taking $F=0$ and $f=\chi_{\dom_0}h$ gives well-posedness of \eqref{eq:stoch-heat-control} for every
$h\in L^2_{\fil}(0,T;L^2(\dom))$.

We now recall a spectral inequality that will be used throughout the analysis. Let $\{e_k\}_{k\ge 1}$ be the eigenfunctions of the Dirichlet Laplacian introduced in \Cref{sec:intro}, and let $\{\tau_k\}_{k\ge 1}$ be the corresponding eigenvalues. For $\lambda>0$, we denote by $N_\lambda$ the number of eigenvalues satisfying $\tau_k\le \lambda$ (recall the definition in \eqref{N_lambda}).

\begin{prop}\label{prop:spectral}
Let $\dom_0 \subset \dom$ be a measurable subset with positive Lebesgue measure. Then there exists a constant $C\geq 1$, depending only on $\dom$ and $\dom_0$, such that for every $\lambda>0$ and every family $\{a_i\}_{i=1}^{N_\lambda}\subset\mathbb R$, it holds
\begin{equation}\label{eq:spectral_ineq}
\left\|\sum_{i=1}^{N_\lambda} a_i e_i\right\|_{L^2(\dom_0)}^2
\ge C^{-1} e^{-C\sqrt{\lambda}} \sum_{i=1}^{N_\lambda} a_i^2.
\end{equation}
\end{prop}

This result follows from \cite[Theorem~5]{AEWZ} and can be viewed as an extension of the classical Lebeau--Robbiano spectral inequality (see \cite{LR95}), which is typically formulated for nonempty open subsets of $\dom$. In contrast, the estimate above applies to measurable subsets of positive Lebesgue measure.

To apply \cite[Theorem~5]{AEWZ}, choose $x_0\in\dom$ and $R\in(0,1]$ such that
\begin{equation*}
B_{4R}(x_0)\subset\dom,
\qquad
|\dom_0\cap B_R(x_0)|>0.
\end{equation*}
Such a ball exists since $|\dom_0|>0$. Applying the theorem with
$E=\dom_0\cap B_R(x_0)$ and
$v=\sum_{i=1}^{N_\lambda}a_i e_i$, we obtain
\begin{equation*}
\|v\|_{L^2(\dom)}
\leq C e^{C\sqrt{\lambda}}\|v\|_{L^1(E)}
\leq C e^{C\sqrt{\lambda}}|E|^{1/2}
\|v\|_{L^2(\dom_0)}.
\end{equation*}
Squaring and using the orthonormality of the eigenfunctions yields
\eqref{eq:spectral_ineq}, after adjusting $C$.

For later use, we introduce the matrix
\begin{equation}\label{def_JN}
J_{N_\lambda} := \big((e_i,e_j)_{L^2(\dom_0)}\big)_{i,j=1}^{N_\lambda},
\end{equation}
which represents the Gram matrix of the first $N_\lambda$ eigenfunctions restricted to $\dom_0$. In view of \Cref{prop:spectral}, this matrix is symmetric positive definite and satisfies
\begin{equation}\label{eq:JN_bound}
\xi^\top J_{N_\lambda}\,\xi
\ge C^{-1} e^{-C\sqrt{\lambda}} \|\xi\|^2,
\quad \forall \xi\in\mathbb R^{N_\lambda},
\end{equation}
where $\|\cdot\|$ denotes the usual Euclidean norm in $\mathbb{R}^{N_\lambda}$.

\section{Mean square stabilization}\label{sec:mean_s}

In this section, we prove \Cref{theo:main}.
We recall that the closed-loop system under consideration is
\begin{equation}\label{eq:closed_loop_section3}
\begin{cases}
\d y = \left(\Delta y + cy - \gamma_\lambda \chi_{\dom_0} P_{N_\lambda} y\right)\dt
+ {\displaystyle\sum_{\ell\geq1}\sqrt{q_\ell}\,G_\ell y\,\d\beta_\ell(t)}
& \textnormal{in } Q_T,\\
y = 0 & \textnormal{on } \Sigma_T,\\
y(0) = y_0 & \textnormal{in } \dom,
\end{cases}
\end{equation}
where $\lambda>0$ is a fixed quantity. For simplicity, hereinafter we write
{$N=N_\lambda$ and $J_N=J_{N_\lambda}$.}

{
By \Cref{thm:general_wellposed}, the closed-loop system admits a unique weak solution on every finite interval.
}

We begin by deriving a series of auxiliary estimates that will be used in the proof of \Cref{theo:main}.
To this end, the solution is decomposed along the eigenfunctions of the Laplacian.

Let $\{e_k\}_{k\geq1}$ be the eigenfunctions of the Dirichlet Laplacian introduced in \Cref{sec:intro}, and denote by $y_k(t)=(y(t),e_k)_{L^2(\dom)}$ the corresponding Fourier coefficients.
For each $k\in\mathbb N^*$, we consider the following ordinary SDE
{
\begin{equation}\label{eq:decomp_low}
\begin{cases}
\displaystyle
\d y_k=
\left((c-\tau_k)y_k
-\gamma_\lambda\sum_{i=1}^{N}y_i(e_i,e_k)_{L^2(\dom_0)}\right)\dt
+\sum_{\ell\geq1}\sqrt{q_\ell}\,g_{\ell k}y_k\,\d\beta_\ell(t),
& t\in(0,T),\\
y_k(0)=(y_0,e_k)_{L^2(\dom)}.
\end{cases}
\end{equation}
}
This equation is obtained by projecting \eqref{eq:closed_loop_section3} onto the eigenfunction $e_k$.

The dynamics are split into low- and high-frequency components.
We introduce the notation
\begin{equation}\label{vec_not}
X_N:=\left(\begin{array}{c}y_1\\y_2\\\vdots\\y_N\end{array}\right)
\quad \textnormal{and}\quad
A_N:=-\left(\begin{array}{cccc}
\tau_1&&&\\
&\tau_2&&\\
&&\ddots&\\
&&&\tau_N
\end{array}\right),
\end{equation}
which collects the low-frequency modes of the solution. In particular, note that
\[
\|P_Ny(t)\|_{L^2(\dom)}^2
=\sum_{k=1}^{N}|y_k(t)|^2
=\|X_N(t)\|^2.
\]

{
Before estimating the low- and high-frequency components, we establish
a martingale property that will allow us to take expectations
in the corresponding energy identities.
}

{
\begin{lem}\label{lem:energy_martingale}
Let $T>0$, $\lambda\geq0$, and
$v\in L^2_{\fil}(\Omega;C([0,T];L^2(\dom)))$.
Then
\begin{equation}\label{eq:energy_martingale}
M_v(t):=
2\sum_{\ell\geq1}\sqrt{q_\ell}
\int_0^t e^{\lambda s}
(v(s),G_\ell v(s))_{L^2(\dom)}\,\d\beta_\ell(s)
\end{equation}
is a martingale on $[0,T]$, and $\esp M_v(t)=0$ for every
$t\in[0,T]$.
\end{lem}
\begin{proof}
By Cauchy--Schwarz inequality, \eqref{eq:noise_bound}, and the continuity of $v$, we have
\begin{equation*}
\int_0^T e^{2\lambda s}\sum_{\ell\geq1}q_\ell
|(v(s),G_\ell v(s))_{L^2(\dom)}|^2\,\d s
\leq e^{2\lambda T}TC_G
\sup_{s\in[0,T]}|v(s)|_{L^2(\dom)}^4<\infty
\quad\text{a.s.}
\end{equation*}
Hence, $M_v$ is a continuous local martingale. By \cite[Proposition D.0.1(ii)]{LR15},
it suffices to verify that the square root of its quadratic
variation at time $T$ has finite expectation. More precisely,
\begin{equation*}
\esp\left(
4\int_0^T e^{2\lambda s}\sum_{\ell\geq1}q_\ell
|(v(s),G_\ell v(s))_{L^2(\dom)}|^2\,\d s
\right)^{1/2}
\leq 2e^{\lambda T}\sqrt{TC_G}\,
\esp\sup_{s\in[0,T]}|v(s)|_{L^2(\dom)}^2<\infty.
\end{equation*}
Thus, $M_v$ is a martingale and, since $M_v(0)=0$,
$\esp M_v(t)=0$ for every $t\in[0,T]$.
\end{proof}
}

We now estimate the low-frequency component.

\begin{lem}\label{lem:lower}
Let {$\lambda>0$ and $\gamma_\lambda>0$} be constants.
Then, for the vector defined in \eqref{vec_not}, we have
{
\begin{align}\label{eq:est_lower}
\esp\left(e^{\lambda t}\|X_N(t)\|^2\right)
&\leq \esp\left(\|X_N(0)\|^2\right)\notag\\
&\quad+
\esp\left(
\int_0^t e^{\lambda s}
\left[\lambda+C_G+2c-2\tau_1
-2C^{-1}e^{-C\sqrt{\lambda}}\gamma_\lambda\right]
\|X_N(s)\|^2\,\d s
\right),
\qquad t\geq0.
\end{align}
}
\end{lem}

\begin{proof}
By It\^o's formula applied to
{$e^{\lambda t}y_k^2$},
where $y_k$ satisfies \eqref{eq:decomp_low}, we have
{
\begin{align*}
\d\left(e^{\lambda t}y_k^2\right)
&=
e^{\lambda t}y_k^2
\left(\lambda+2c+\sum_{\ell\geq1}q_\ell|g_{\ell k}|^2\right)\dt\\
&\quad
-2\gamma_\lambda e^{\lambda t}
\sum_{i=1}^{N}y_ky_i(e_i,e_k)_{L^2(\dom_0)}\dt
-2\tau_ke^{\lambda t}y_k^2\dt\\
&\quad+
2e^{\lambda t}\sum_{\ell\geq1}
\sqrt{q_\ell}\,g_{\ell k}y_k^2\,\d\beta_\ell(t).
\end{align*}
}
{
Summing over $k=1,\ldots,N$ and using
\eqref{def_JN} and \eqref{vec_not}, we obtain
\begin{align*}
\d\left(e^{\lambda t}|X_N|^2\right)
&=
e^{\lambda t}
\left[(\lambda+2c)|X_N|^2+2X_N^TA_NX_N\right]\dt
-2e^{\lambda t}\gamma_\lambda X_N^TJ_NX_N\dt\\
&\quad+
e^{\lambda t}\sum_{k=1}^{N}
\left(\sum_{\ell\geq1}q_\ell|g_{\ell k}|^2\right)y_k^2\dt
+\d M_{P_Ny}(t),
\end{align*}
where we have used \eqref{eq:energy_martingale}
and \eqref{eq:diagonal_noise} to write
\begin{align*}
M_{P_Ny}(t)
&=2\sum_{\ell\geq1}\sqrt{q_\ell}
\int_0^t e^{\lambda s}
(P_Ny(s),G_\ell P_Ny(s))_{L^2(\dom)}\,\d\beta_\ell(s) \\
&=2\sum_{\ell\geq1}\sqrt{q_\ell}
\int_0^t e^{\lambda s}
\sum_{k=1}^{N}g_{\ell k}y_k(s)^2\,\d\beta_\ell(s).
\end{align*}
}
Using \eqref{eq:JN_bound}, we have
$X_N^TJ_NX_N\geq C^{-1}e^{-C\sqrt{\lambda}}\|X_N\|^2$.
{
Moreover, \eqref{eq:noise_bound} gives
\begin{equation*}
\sum_{k=1}^{N}
\left(\sum_{\ell\geq1}q_\ell|g_{\ell k}|^2\right)y_k^2
\leq C_G\|X_N\|^2.
\end{equation*}
}
From the definition of $A_N$ in \eqref{vec_not}, we have
$X_N^TA_NX_N\leq-\tau_1\|X_N\|^2$, which yields
{
\begin{align*}
\d\left(e^{\lambda t}\|X_N\|^2\right)
&\leq
e^{\lambda t}
\left(\lambda+C_G+2c-2\tau_1
-2C^{-1}e^{-C\sqrt{\lambda}}\gamma_\lambda\right)
\|X_N\|^2\dt
+\d M_{P_Ny}(t).
\end{align*}
By \Cref{thm:general_wellposed},
$P_Ny\in L^2_{\fil}(\Omega;C([0,T];L^2(\dom)))$
for every $T>0$.
Thus, \Cref{lem:energy_martingale} gives
$\esp M_{P_Ny}(t)=0$.
Integrating over $[0,t]$ and taking expectation yields
\eqref{eq:est_lower}.
}
\end{proof}

We now turn to the high-frequency part of the solution. For the remaining modes, $k\geq N+1$, we define
\begin{equation}\label{rem:modes}
z=\sum_{k=N+1}^{\infty}y_ke_k=P_N^\bot y.
\end{equation}

\begin{lem}\label{lem:solu_high}
The stochastic process defined in \eqref{rem:modes} satisfies
\begin{equation}\label{eq:proy_high}
\begin{cases}
\d z=
\left[\Delta z+cz
-\gamma_\lambda P_N^\bot\left(\chi_{\dom_0}P_Ny\right)\right]\dt
+{\displaystyle\sum_{\ell\geq1}
\sqrt{q_\ell}\,G_\ell z\,\d\beta_\ell(t)}
& \textnormal{in }Q_T,\\
z=0 & \textnormal{on }\Sigma_T,\\
z(0)=P_N^\bot y_0 & \textnormal{in }\dom,
\end{cases}
\end{equation}
in the sense of \Cref{def:weak_solution}.
\end{lem}

{
\begin{proof}
Let $\phi\in H_0^1(\dom)$. Since
$P_N^\bot\phi\in H_0^1(\dom)$, we may use it as a test
function in the weak formulation of
\eqref{eq:closed_loop_section3}.
Since $P_N^\bot$ is self-adjoint, we have
\begin{equation*}
(y,P_N^\bot\phi)_{L^2(\dom)}
=(z,\phi)_{L^2(\dom)}.
\end{equation*}
Moreover, using $\phi=P_N\phi+P_N^\bot\phi$ and integration by parts, we obtain
\begin{align*}
(\nabla y,\nabla P_N^\bot\phi)_{L^2(\dom)}
&=(\nabla y,\nabla\phi)_{L^2(\dom)}
+(y,\Delta P_N\phi)_{L^2(\dom)}\\
&=(\nabla y,\nabla\phi)_{L^2(\dom)}
+(\Delta P_Ny,\phi)_{L^2(\dom)}\\
&=(\nabla(y-P_Ny),\nabla\phi)_{L^2(\dom)}
=(\nabla z,\nabla\phi)_{L^2(\dom)}.
\end{align*}
For the stochastic term, \eqref{eq:diagonal_noise} yields
\begin{equation*}
(G_\ell y,P_N^\bot\phi)_{L^2(\dom)}
=(P_N^\bot G_\ell y,\phi)_{L^2(\dom)}
=(G_\ell z,\phi)_{L^2(\dom)},
\qquad \ell\geq1.
\end{equation*}
Therefore,
\begin{align*}
(z(t),\phi)_{L^2(\dom)}
&=(P_N^\bot y_0,\phi)_{L^2(\dom)}
-\int_0^t(\nabla z(s),\nabla\phi)_{L^2(\dom)}\,\d s
+c\int_0^t(z(s),\phi)_{L^2(\dom)}\,\d s\\
&\quad-\gamma_\lambda\int_0^t
(P_N^\bot(\chi_{\dom_0}P_Ny(s)),\phi)_{L^2(\dom)}\,\d s +\sum_{\ell\geq1}\sqrt{q_\ell}
\int_0^t(G_\ell z(s),\phi)_{L^2(\dom)}\,\d\beta_\ell(s),
\end{align*}
almost surely for all $t\in[0,T]$. Thus, $z$ satisfies \eqref{eq:proy_high} in the sense of
\Cref{def:weak_solution}.
\end{proof}
}

We next derive an estimate for the high-frequency part.
\begin{lem}\label{lem:high}
For any {$\epsilon>0$ and $\lambda>0$},
the solution $z$ to \eqref{eq:proy_high} satisfies
{
\begin{align}\label{eq:est_high}
\esp\left(e^{\lambda t}\|z(t)\|_{L^2(\dom)}^2\right)
&\leq
\esp\left(\|z(0)\|_{L^2(\dom)}^2\right)\notag +
(C_G+2c-\lambda+\epsilon)
\esp\left(\int_0^t e^{\lambda s}
\|z(s)\|_{L^2(\dom)}^2\,\d s\right)\notag\\
&\quad+
\frac{\gamma_\lambda^2}{\epsilon}
\esp\left(\int_0^t e^{\lambda s}
\|X_N(s)\|^2\,\d s\right),
\qquad t\geq0.
\end{align}
}
\end{lem}

\begin{proof}
Applying It\^o's formula to the mapping
{$t\mapsto e^{\lambda t}\|z(t)\|_{L^2(\dom)}^2$}
and integrating over $[0,t]$, we obtain
{
\begin{align*}
e^{\lambda t}\|z(t)\|_{L^2(\dom)}^2
&=\|z(0)\|_{L^2(\dom)}^2
+(\lambda+2c)\int_0^t e^{\lambda s}
\|z(s)\|_{L^2(\dom)}^2\,\d s -
2\int_0^t e^{\lambda s}
\|\nabla z(s)\|_{L^2(\dom)}^2\,\d s\\
&\quad-
2\gamma_\lambda\int_0^t e^{\lambda s}
\left(z(s),P_N^\bot(\chi_{\dom_0}P_Ny(s))\right)_{L^2(\dom)}
\,\d s\\
&\quad+
\int_0^t e^{\lambda s}
\sum_{\ell\geq1}q_\ell
\|G_\ell z(s)\|_{L^2(\dom)}^2\,\d s
+M_z(t),
\end{align*}
where $M_z$ is defined using  the notation \eqref{eq:energy_martingale}.
}
Using the spectral decomposition \eqref{rem:modes}, we have
{
\begin{equation*}
\|\nabla z(s)\|_{L^2(\dom)}^2
=\sum_{k=N+1}^{\infty}\tau_k|y_k(s)|^2
\geq\lambda\|z(s)\|_{L^2(\dom)}^2
\quad\textnormal{for a.e. }s,\ \textnormal{a.s.}
\end{equation*}
This, together with \eqref{eq:noise_bound}, gives
\begin{align*}
e^{\lambda t}\|z(t)\|_{L^2(\dom)}^2
&\leq \|z(0)\|_{L^2(\dom)}^2
+(C_G+2c-\lambda)\int_0^t e^{\lambda s}
\|z(s)\|_{L^2(\dom)}^2\,\d s\\
&\quad+
\left|2\gamma_\lambda\int_0^t e^{\lambda s}
\left(z(s),P_N^\bot(\chi_{\dom_0}P_Ny(s))\right)_{L^2(\dom)}
\,\d s\right|
+M_z(t).
\end{align*}
}
Using Cauchy--Schwarz inequality and that
$\|P_N^\bot(\chi_{\dom_0}P_Ny)\|_{L^2(\dom)}
\leq\|P_Ny\|_{L^2(\dom)}$, we obtain
{
\begin{align*}
e^{\lambda t}\|z(t)\|_{L^2(\dom)}^2
&\leq \|z(0)\|_{L^2(\dom)}^2
+(C_G+2c-\lambda)\int_0^t e^{\lambda s}
\|z(s)\|_{L^2(\dom)}^2\,\d s\\
&\quad+
2\gamma_\lambda\int_0^t e^{\lambda s}
\|z(s)\|_{L^2(\dom)}
\|P_Ny(s)\|_{L^2(\dom)}\,\d s
+M_z(t).
\end{align*}
}
Applying Young's inequality, for any $\epsilon>0$, we deduce
{
\begin{align*}
e^{\lambda t}\|z(t)\|_{L^2(\dom)}^2
&\leq \|z(0)\|_{L^2(\dom)}^2
+(C_G+2c-\lambda+\epsilon)
\int_0^t e^{\lambda s}\|z(s)\|_{L^2(\dom)}^2\,\d s\\
&\quad+
\frac{\gamma_\lambda^2}{\epsilon}
\int_0^t e^{\lambda s}\|P_Ny(s)\|_{L^2(\dom)}^2\,\d s
+M_z(t).
\end{align*}
By \Cref{thm:general_wellposed}, we have $z=P_N^\bot y\in
L^2_{\fil}(\Omega;C([0,T];L^2(\dom)))$ for every $T>0$. Thus, \Cref{lem:energy_martingale} gives $\esp M_z(t)=0$.
Taking expectation yields \eqref{eq:est_high}.
}
\end{proof}

We are now in a position to prove our main theorem.

\begin{proof}[Proof of \Cref{theo:main}]
Let $\lambda>\max\{2\tau_1,{C_G}+2c+\tau_1\}$ and set
\begin{equation}\label{def:gamma_lambda}
\gamma_\lambda=Ce^{C\sqrt{\lambda}}\lambda,
\end{equation}
where $C\geq1$ is the constant appearing in \eqref{eq:JN_bound}.

{
In \Cref{lem:high}, choose
\begin{equation*}
\epsilon=\lambda-C_G-2c.
\end{equation*}
}
By the choice of $\lambda$, we have
\begin{equation*}
\epsilon>\tau_1>0.
\end{equation*}

Let $\mu_\lambda\geq1$ be a constant to be fixed later.
{Multiplying \eqref{eq:est_lower} by $\mu_\lambda$
and adding \eqref{eq:est_high},}
using that $\|X_N\|=\|P_Ny\|_{L^2(\dom)}$ and $z=P_N^\bot y$, we obtain
{
\begin{align}\label{add_estimates}
&\mu_\lambda\esp\left(e^{\lambda t}\|P_Ny(t)\|_{L^2(\dom)}^2\right)
+\esp\left(e^{\lambda t}\|P_N^\bot y(t)\|_{L^2(\dom)}^2\right)
\notag\\
&\quad\leq
\mu_\lambda\esp\left(\|P_Ny_0\|_{L^2(\dom)}^2\right)
+\esp\left(\|P_N^\bot y_0\|_{L^2(\dom)}^2\right)
\notag\\
&\qquad+
\mu_\lambda
\left[\lambda+C_G+2c-2\tau_1
-2C^{-1}e^{-C\sqrt{\lambda}}\gamma_\lambda\right]
\esp\left(\int_0^t e^{\lambda s}
\|P_Ny(s)\|_{L^2(\dom)}^2\,\d s\right)
\notag\\
&\qquad+
(C_G+2c-\lambda+\epsilon)
\esp\left(\int_0^t e^{\lambda s}
\|P_N^\bot y(s)\|_{L^2(\dom)}^2\,\d s\right)
\notag\\
&\qquad+
\frac{\gamma_\lambda^2}{\epsilon}
\esp\left(\int_0^t e^{\lambda s}
\|P_Ny(s)\|_{L^2(\dom)}^2\,\d s\right).
\end{align}
}

By the choice of $\gamma_\lambda$, we have
$2C^{-1}e^{-C\sqrt{\lambda}}\gamma_\lambda=2\lambda$,
and therefore
{
\begin{equation*}
\lambda+C_G+2c-2\tau_1
-2C^{-1}e^{-C\sqrt{\lambda}}\gamma_\lambda
=-(\lambda+2\tau_1-C_G-2c).
\end{equation*}
}
Moreover,
{
\begin{equation*}
C_G+2c-\lambda+\epsilon=0.
\end{equation*}
}

Hence \eqref{add_estimates} reduces to
{
\begin{align}\label{add_estimates_2}
&\mu_\lambda\esp\left(e^{\lambda t}\|P_Ny(t)\|^2\right)
+\esp\left(e^{\lambda t}\|P_N^\bot y(t)\|^2\right)
\notag\\
&\quad\leq
\mu_\lambda\esp\left(\|P_Ny_0\|^2\right)
+\esp\left(\|P_N^\bot y_0\|^2\right)
\notag\\
&\qquad+
\left[-\mu_\lambda(\lambda+2\tau_1-C_G-2c)
+\frac{\gamma_\lambda^2}{\epsilon}\right]
\esp\left(\int_0^t e^{\lambda s}\|P_Ny(s)\|^2\,\d s\right).
\end{align}
}

We choose $\mu_\lambda:=\max\left\{1,\frac{\gamma_\lambda^2}{\epsilon^2}\right\}$. Then $\mu_\lambda\geq1$ and, since
{$\lambda+2\tau_1-C_G-2c=\epsilon+2\tau_1>\epsilon$,}
we have
{
\begin{equation*}
-\mu_\lambda(\lambda+2\tau_1-C_G-2c)
+\frac{\gamma_\lambda^2}{\epsilon}
\leq0.
\end{equation*}
}
Hence, \eqref{add_estimates_2} reduces to
{
\begin{align*}
&\mu_\lambda\esp\left(e^{\lambda t}\|P_Ny(t)\|^2\right)
+\esp\left(e^{\lambda t}\|P_N^\bot y(t)\|^2\right)\\
&\qquad\leq
\mu_\lambda\esp\left(\|P_Ny_0\|^2\right)
+\esp\left(\|P_N^\bot y_0\|^2\right).
\end{align*}
}

Using the orthogonality of the eigenfunctions and the fact that $\mu_\lambda\geq1$, we obtain
{
\begin{equation*}
\esp\left(e^{\lambda t}\|y(t)\|_{L^2(\dom)}^2\right)
\leq\mu_\lambda\esp\left(\|y_0\|_{L^2(\dom)}^2\right).
\end{equation*}
}
{Therefore,}
\begin{equation}\label{eq:est_final_unified}
\esp\left(\|y(t)\|_{L^2(\dom)}^2\right)
\leq\mu_\lambda e^{-\lambda t}
\esp\left(\|y_0\|_{L^2(\dom)}^2\right).
\end{equation}

It remains to estimate $\mu_\lambda$. Since $\epsilon>\tau_1$, we have
\begin{equation*}
\mu_\lambda
\leq1+\frac{\gamma_\lambda^2}{\epsilon^2}
\leq Ce^{C\sqrt{\lambda}}\lambda^2.
\end{equation*}
Moreover, the function $\lambda\mapsto\lambda^2e^{-\alpha\sqrt{\lambda}}$
is bounded on $[0,\infty)$ for any $\alpha>0$.
Hence, the polynomial factor can be absorbed into the exponential
(up to redefining $C$), and we obtain
\begin{equation*}
\mu_\lambda\leq Ce^{C\sqrt{\lambda}}.
\end{equation*}
Substituting into \eqref{eq:est_final_unified}, we conclude that
\begin{equation}\label{eq:est_final_case_unified}
\esp\left(\|y(t)\|_{L^2(\dom)}^2\right)
\leq Ce^{C\sqrt{\lambda}}e^{-\lambda t}
\esp\left(\|y_0\|_{L^2(\dom)}^2\right),
\qquad \forall t\geq0.
\end{equation}

Finally, using \eqref{eq:def_feedback_intro}, we have
\begin{equation*}
\esp\left(\|\mathscr H_\lambda y(t)\|_{L^2(\dom)}^2\right)
=\gamma_\lambda^2\esp\left(\|P_Ny(t)\|_{L^2(\dom)}^2\right)
\leq\gamma_\lambda^2\esp\left(\|y(t)\|_{L^2(\dom)}^2\right).
\end{equation*}
{
Combining \eqref{eq:est_final_case_unified} with
\eqref{def:gamma_lambda} and absorbing the polynomial factor
$\lambda^2$ into the exponential, we obtain
}
\begin{equation*}
\esp\left(\|\mathscr H_\lambda y(t)\|_{L^2(\dom)}^2\right)
\leq Ce^{C\sqrt{\lambda}}e^{-\lambda t}
\esp\left(\|y_0\|_{L^2(\dom)}^2\right),
\qquad \forall t\geq0.
\end{equation*}

{
It remains to verify the asserted integrability on $[0,\infty)$.
Applying It\^o's formula to $\|y(t)\|_{L^2(\dom)}^2$ and integrating
over $[0,T]$, we obtain
\begin{align*}
\|y(T)\|_{L^2(\dom)}^2&
+2\int_0^T\|\nabla y(t)\|_{L^2(\dom)}^2\,\dt \\
&=\|y_0\|_{L^2(\dom)}^2
+2c\int_0^T\|y(t)\|_{L^2(\dom)}^2\,\dt -2\gamma_\lambda\int_0^T
(y(t),\chi_{\dom_0}P_Ny(t))_{L^2(\dom)}\,\dt\\
&\quad+
\int_0^T\sum_{\ell\geq1}q_\ell
\|G_\ell y(t)\|_{L^2(\dom)}^2\,\dt +
2\sum_{\ell\geq1}\sqrt{q_\ell}
\int_0^T(y(t),G_\ell y(t))_{L^2(\dom)}\,\d\beta_\ell(t).
\end{align*}
The last term has zero expectation by
\Cref{lem:energy_martingale}. Using \eqref{eq:noise_bound} and
$|(y,\chi_{\dom_0}P_Ny)_{L^2(\dom)}|
\leq\|y\|_{L^2(\dom)}^2$, we deduce
\begin{align*}
2\esp\int_0^T\|\nabla y(t)\|_{L^2(\dom)}^2\,\dt \leq
\esp\|y_0\|_{L^2(\dom)}^2 +
(2|c|+C_G+2\gamma_\lambda)
\int_0^T\esp\|y(t)\|_{L^2(\dom)}^2\,\dt.
\end{align*}

Similarly, we can deduce that
\begin{align*}
\esp\sup_{t\in[0,T]}\|y(t)\|_{L^2(\dom)}^2
&\leq C\esp\|y_0\|_{L^2(\dom)}^2
+C(1+\gamma_\lambda)
\int_0^T\esp\|y(t)\|_{L^2(\dom)}^2\,\dt,
\end{align*}
where $C$ is independent of $T$.
By \eqref{eq:est_final_case_unified}, the right-hand sides of both
estimates are bounded uniformly in $T$.
Letting $T\to\infty$ and using monotone convergence proves
\begin{equation*}
\esp\sup_{t\geq0}\|y(t)\|_{L^2(\dom)}^2
+\esp\int_0^\infty\|\nabla y(t)\|_{L^2(\dom)}^2\,\dt
<\infty.
\end{equation*}
}
This concludes the proof.
\end{proof}

\section{Almost sure exponential stability}\label{sec:as_s}

The aim of this section is to {deduce almost sure exponential stability from the mean-square exponential stability result obtained in \Cref{theo:main}}. The strategy relies on combining uniform-in-time estimates in expectation with a probabilistic argument of Borel--Cantelli type.

We begin with an auxiliary estimate.
\begin{lem}\label{lem:BDG_sup}
Let $y$ be the solution to \eqref{eq:forward_semilinear_feed_intro}.
Then, for any $0\leq T_0<T$ and any $\epsilon>0$, there exists
a constant $C>0$, independent of $T_0$, $T$, and $\epsilon$, such that
{
\begin{align}\label{eq:BDG_est}
&\esp\left(\sup_{t\in[T_0,T]}
\left|\sum_{\ell\geq1}\sqrt{q_\ell}
\int_{T_0}^t(y(s),G_\ell y(s))_{L^2(\dom)}\,\d\beta_\ell(s)\right|\right)
\notag\\
&\qquad\leq
\frac{\epsilon}{2}\,
\esp\left(\sup_{s\in[T_0,T]}\|y(s)\|^2_{L^2(\dom)}\right)
+\frac{C}{2\epsilon}
\int_{T_0}^T\esp\|y(s)\|^2_{L^2(\dom)}\,\d s.
\end{align}
}
\end{lem}

{
\begin{proof}
By \cite[Proposition D.0.1(i)]{LR15}, Cauchy--Schwarz,
and \eqref{eq:est_ruido}, we have
\begin{align}\label{bdg_estimate}
&\esp\left(\sup_{t\in[T_0,T]}
\left|\sum_{\ell\geq1}\sqrt{q_\ell}
\int_{T_0}^t(y(s),G_\ell y(s))_{L^2(\dom)}\,\d\beta_\ell(s)\right|\right)
\notag\\
&\qquad\leq
3\esp\left(
\int_{T_0}^T\sum_{\ell\geq1}q_\ell
|(y(s),G_\ell y(s))_{L^2(\dom)}|^2\,\d s
\right)^{1/2}
\notag\\
&\qquad\leq
3\sqrt{C_G}\,
\esp\left[
\sup_{s\in[T_0,T]}\|y(s)\|_{L^2(\dom)}
\left(\int_{T_0}^T\|y(s)\|^2_{L^2(\dom)}\,\d s\right)^{1/2}
\right].
\end{align}
Applying Young's inequality and Fubini's theorem yields
\eqref{eq:BDG_est}.
\end{proof}
}

{
Combining \Cref{lem:BDG_sup} with energy estimates, we obtain
a uniform bound on arbitrary time intervals.
}

\begin{prop}\label{prop:tgeq0}
Under the assumptions of \Cref{theo:main}, there exists a constant
$C>0$, depending at most on {$C_G$}, $c$, $\dom$,
and $\dom_0$, such that for any $0\leq T_0<T$, the solution to
\eqref{eq:forward_semilinear_feed_intro} satisfies
\begin{equation}\label{boud:sup}
\esp\left(\sup_{t\in[T_0,T]}\|y(t)\|^2_{L^2(\dom)}\right)
\leq C\esp\left(\|y(T_0)\|^2_{L^2(\dom)}\right)
+C\gamma_{\lambda}\int_{T_0}^T
\esp\|y(s)\|^2_{L^2(\dom)}\,\d s.
\end{equation}
\end{prop}
\begin{proof}
Let $0\leq T_0<T$ be fixed. Applying It\^{o}'s formula to
$t\mapsto\|y(t)\|^2_{L^2(\dom)}$, and using integration by parts,
we obtain
\begin{align*}
\|y(t)\|^2_{L^2(\dom)}
&=\|y(T_0)\|^2_{L^2(\dom)}
-2\int_{T_0}^t\|\nabla y(s)\|^2_{L^2(\dom)}\,\d s
+2c\int_{T_0}^t\|y(s)\|^2_{L^2(\dom)}\,\d s\\
&\quad+2\int_{T_0}^t
\left(y(s),{\chi_{\dom_0}}\mathscr H_{\lambda}y(s)\right)_{L^2(\dom)}
\,\d s
+{\int_{T_0}^t\sum_{\ell\geq1}q_\ell
\|G_\ell y(s)\|^2_{L^2(\dom)}\,\d s}\\
&\quad+{2\sum_{\ell\geq1}\sqrt{q_\ell}
\int_{T_0}^t(y(s),G_\ell y(s))_{L^2(\dom)}\,\d\beta_\ell(s)},
\end{align*}
almost surely for all $t\in[T_0,T]$.
By \eqref{eq:def_feedback_intro},
\begin{equation*}
\left|\left(y,{\chi_{\dom_0}}\mathscr H_\lambda y\right)_{L^2(\dom)}\right|
\leq \|y\|_{L^2(\dom)}
\|{\chi_{\dom_0}}\mathscr H_\lambda y\|_{L^2(\dom)}
\leq \gamma_\lambda\|y\|^2_{L^2(\dom)}.
\end{equation*}
Taking the supremum over $t\in[T_0,T]$, dropping the negative
gradient term, and using \eqref{eq:est_ruido}, we obtain
\begin{align*}
\sup_{t\in[T_0,T]}\|y(t)\|^2_{L^2(\dom)}
&\leq \|y(T_0)\|^2_{L^2(\dom)}
+\left(2\gamma_\lambda+2|c|+{C_G}\right)
\int_{T_0}^T\|y(s)\|^2_{L^2(\dom)}\,\d s\\
&\quad+{2}\sup_{t\in[T_0,T]}
\left|{\sum_{\ell\geq1}\sqrt{q_\ell}
\int_{T_0}^t(y(s),G_\ell y(s))_{L^2(\dom)}\,\d\beta_\ell(s)}\right|,
\qquad\text{a.s.}
\end{align*}
Taking expectation and applying \Cref{lem:BDG_sup}, we obtain,
for any $\varepsilon>0$,
\begin{align*}
\esp\left(\sup_{t\in[T_0,T]}\|y(t)\|^2_{L^2(\dom)}\right)
&\leq \esp\left(\|y(T_0)\|^2_{L^2(\dom)}\right)
+{\varepsilon}\,
\esp\left(\sup_{t\in[T_0,T]}\|y(t)\|^2_{L^2(\dom)}\right)\\
&\quad+\left(2\gamma_\lambda+2|c|+{C_G+\frac{C}{\varepsilon}}\right)
\int_{T_0}^T\esp\|y(s)\|^2_{L^2(\dom)}\,\d s.
\end{align*}
Choosing $\varepsilon=1/2$ and absorbing the corresponding term
into the left-hand side gives
\begin{equation*}
\esp\left(\sup_{t\in[T_0,T]}\|y(t)\|^2_{L^2(\dom)}\right)
\leq C\esp\left(\|y(T_0)\|^2_{L^2(\dom)}\right)
+C(\gamma_\lambda+1)\int_{T_0}^T
\esp\|y(s)\|^2_{L^2(\dom)}\,\d s.
\end{equation*}
Finally, recalling \eqref{def:gamma_lambda} and the condition
$\lambda>2\tau_1$ (where $\tau_1$ depends only on $\dom$),
the term $\gamma_\lambda+1$ can be bounded by $C\gamma_\lambda$
for some $C>0$, which yields \eqref{boud:sup}.
\end{proof}

{
Combining \Cref{prop:tgeq0} with the mean-square decay in
\Cref{theo:main}, we obtain the following estimate on unbounded
time intervals.
}

\begin{cor}\label{cor:tgeqT0}
Under the assumptions of \Cref{theo:main}, there exists a constant
$C>0$, depending at most on {$C_G$}, $c$, $\dom$,
and $\dom_0$, such that for any $T_0\geq0$, the solution to
\eqref{eq:forward_semilinear_feed_intro} satisfies
\begin{equation}\label{eq:sup_from_T0}
\esp\left(\sup_{t\geq T_0}\|y(t)\|^2_{L^2(\dom)}\right)
\leq C e^{C\sqrt{\lambda}}
{e^{-\lambda T_0}}
\esp\left(\|y_0\|^2_{L^2(\dom)}\right).
\end{equation}
\end{cor}

\begin{proof}
Let $T>T_0$. Applying \Cref{prop:tgeq0}, we obtain
\begin{equation}\label{eq:est_T0}
\esp\left(\sup_{t\in[T_0,T]}\|y(t)\|^2_{L^2(\dom)}\right)
\leq C\esp\left(\|y(T_0)\|^2_{L^2(\dom)}\right)
+C\gamma_\lambda\int_{T_0}^T
\esp\|y(s)\|^2_{L^2(\dom)}\,\d s.
\end{equation}
{
By \Cref{theo:main},
\begin{align*}
\esp\left(\|y(T_0)\|^2_{L^2(\dom)}\right)
&\leq Ce^{C\sqrt{\lambda}}e^{-\lambda T_0}
\esp\left(\|y_0\|^2_{L^2(\dom)}\right),
\end{align*}
and
\begin{align*}
\int_{T_0}^T\esp\|y(s)\|^2_{L^2(\dom)}\,\d s
&\leq Ce^{C\sqrt{\lambda}}
\esp\left(\|y_0\|^2_{L^2(\dom)}\right)
\int_{T_0}^T e^{-\lambda s}\,\d s\\
&\leq \frac{C}{\lambda}e^{C\sqrt{\lambda}}e^{-\lambda T_0}
\esp\left(\|y_0\|^2_{L^2(\dom)}\right).
\end{align*}
Recalling \eqref{def:gamma_lambda}, we deduce from \eqref{eq:est_T0} that
\begin{equation*}
\esp\left(\sup_{t\in[T_0,T]}\|y(t)\|^2_{L^2(\dom)}\right)
\leq Ce^{C\sqrt{\lambda}}e^{-\lambda T_0}
\esp\left(\|y_0\|^2_{L^2(\dom)}\right),
\end{equation*}
}
where $C>0$ is independent of $T$.
Letting $T\to\infty$ and using monotone convergence yields
\eqref{eq:sup_from_T0}.
\end{proof}

We are now in a position to prove the almost sure exponential stability
announced in the introduction.
{
The proof combines \Cref{cor:tgeqT0} with Markov's inequality
and the Borel--Cantelli lemma.
}

\begin{proof}[Proof of \Cref{thm:as_intro}]
{
If $\esp\|y_0\|^2_{L^2(\dom)}=0$, the solution is identically zero
and the conclusion follows, with the convention $\log 0=-\infty$.
We therefore assume that $\esp\|y_0\|^2_{L^2(\dom)}>0$.

Let $N\geq1$ be an integer and $\eta>0$.
By Markov's inequality and \Cref{cor:tgeqT0},
\begin{align}\label{eq:est_prob_1_as}
\mathbb P\left(
\sup_{t\in[N,N+1]}\|y(t)\|^2_{L^2(\dom)}\geq\eta^2
\right)
&\leq \frac{1}{\eta^2}
\esp\left(\sup_{t\geq N}\|y(t)\|^2_{L^2(\dom)}\right)\notag\\
&\leq \frac{C}{\eta^2}e^{C\sqrt{\lambda}}e^{-\lambda N}
\esp\left(\|y_0\|^2_{L^2(\dom)}\right),
\end{align}
where $C>0$ is independent of $N$ and $\lambda$.
}

We now choose $\eta=\eta_N$ given by
\begin{equation*}
\eta_N^2:=
Ce^{C\sqrt{\lambda}}\esp\left(\|y_0\|^2_{L^2(\dom)}\right)
N^2e^{-\lambda N},
\end{equation*}
where $C$ is the constant appearing in \eqref{eq:est_prob_1_as}.
Then \eqref{eq:est_prob_1_as} yields
\begin{equation*}
\mathbb P\left(
\sup_{t\in[N,N+1]}\|y(t)\|^2_{L^2(\dom)}\geq\eta_N^2
\right)\leq\frac{1}{N^2}.
\end{equation*}
Since $\sum_{N=1}^{\infty}N^{-2}<\infty$,
the Borel--Cantelli lemma implies that for almost every
$\omega\in\Omega$ there exists $N_0=N_0(\omega)\geq1$
such that for all $N\geq N_0$,
\begin{equation*}
\sup_{t\in[N,N+1]}\|y(t)\|^2_{L^2(\dom)}
\leq Ce^{C\sqrt{\lambda}}
\esp\left(\|y_0\|^2_{L^2(\dom)}\right)N^2e^{-\lambda N}.
\end{equation*}

{
For $t\in[N,N+1]$ with $N\geq N_0(\omega)$, taking logarithms
and dividing by $t$ gives
\begin{equation*}
\frac{1}{t}\log\|y(t)\|^2_{L^2(\dom)}
\leq
\frac{1}{t}\log\left(
Ce^{C\sqrt{\lambda}}\esp\left(\|y_0\|^2_{L^2(\dom)}\right)
\right)
+\frac{2\log N}{t}-\lambda\frac{N}{t}.
\end{equation*}
Since $N/t\to1$ and $(\log N)/t\to0$ as $t\to\infty$,
}
we conclude that
\begin{equation*}
\limsup_{t\to\infty}\frac{1}{t}
\log\left(\|y(t)\|^2_{L^2(\dom)}\right)
\leq-\lambda
\qquad\textnormal{a.s.}
\end{equation*}
This ends the proof.
\end{proof}

\section{Null-controllability result by a feedback iterative method}\label{sec:nc}

In this section, we prove the null-controllability of system
\eqref{eq:stoch-heat-control} by means of a construction based on feedback laws.

\begin{proof}[Proof of \Cref{thm:null_feed}]
For readability, we split the proof into three steps. Without loss
of generality, we assume that $1/T=n_T\in\N^*$
(see \Cref{rem:time_T_gen}).

\smallskip
\textit{-- Step 1: Construction of the closed-loop solution.}
Let $\Gamma>0$ be a constant, independent of
{$T\in(0,1]$}, to be fixed later,
{such that
\begin{equation*}
\Gamma^2>\max\{2\tau_1,C_G+2c+\tau_1\}.
\end{equation*}
}
We define
\begin{equation*}
T_n:=T-\frac{1}{n}, \qquad
\lambda_n:=\Gamma^2n^4, \qquad n\geq n_T.
\end{equation*}
{
Since $\lambda_n\geq\lambda_{n_T}\geq\Gamma^2$, every $\lambda_n$
satisfies the hypothesis of \Cref{theo:main}.
}
Since $1/T=n_T$, we have $T_{n_T}=0$. We then set
\begin{equation*}
I_n:=[T_n,T_{n+1}), \qquad n\geq n_T.
\end{equation*}

For each $n\geq n_T$, we consider on $I_n$ the controlled system
\begin{equation}\label{eq:forward_semilinear_feed_n}
\begin{cases}
\d y=(\Delta y+cy+\chi_{\dom_0}\mathscr H_{\lambda_n}y)\,\dt
+{\displaystyle\sum_{\ell\geq1}
\sqrt{q_\ell}\,G_\ell y\,\d\beta_\ell(t)}
& \text{in }I_n\times\dom,\\
y=0 & \text{on }I_n\times\Gamma.
\end{cases}
\end{equation}

Since the operator $y\mapsto\chi_{\dom_0}\mathscr H_{\lambda_n}y$
is linear and bounded in $L^2(\dom)$, it is globally Lipschitz.
Therefore, by \Cref{thm:general_wellposed},
{applied on $[T_n,T_{n+1}]$,}
for any initial datum in
$L^2_{\mathcal F_{T_n}}(\Omega;L^2(\dom))$,
system \eqref{eq:forward_semilinear_feed_n} admits a unique weak solution
\begin{equation*}
y\in L^2_{\fil}(\Omega;C({[T_n,T_{n+1}]};L^2(\dom)))
\cap L^2_{\fil}(I_n;H_0^1(\dom)).
\end{equation*}

We construct inductively a process $y$ on $[0,T)$ as follows.
On {$[0,T_{n_T+1}]$}, we solve
\eqref{eq:forward_semilinear_feed_n} with initial condition $y(0)=y_0$.
Assume that $y$ has been constructed on $[0,T_n]$.
Since $y(T_n)\in L^2_{\mathcal F_{T_n}}(\Omega;L^2(\dom))$,
we apply \Cref{thm:general_wellposed} on
{$[T_n,T_{n+1}]$} with initial condition $y(T_n)$
and obtain a unique solution on {$[T_n,T_{n+1}]$}.
Proceeding by induction, we obtain a unique adapted process $y$
on $[0,T)$ associated with the piecewise feedback law
\begin{equation}\label{eq:def_feedback_global}
h(t)=\mathscr H_{\lambda_n}y(t),
\qquad t\in I_n,\ n\geq n_T.
\end{equation}

\smallskip
\textit{-- Step 2: {Decay estimates and integrability of the control.}}
By \Cref{theo:main}, for all
{$t\in[T_n,T_{n+1}]$} and $n\geq n_T$,
the solution satisfies
\begin{equation}\label{eq:yt1}
\esp\|y(t)\|_{L^2(\dom)}^2
\leq Ce^{C\Gamma n^2}
e^{-\frac{\Gamma^2n^4}{2}(t-T_n)}
\esp\|y(T_n)\|_{L^2(\dom)}^2.
\end{equation}
Evaluating \eqref{eq:yt1} at $t=T_{n+1}$ and using that
$T_{n+1}-T_n=\frac{1}{n(n+1)}\geq\frac{1}{2n^2}$, we obtain
\begin{equation}\label{eq:Tn}
\esp\|y(T_{n+1})\|_{L^2(\dom)}^2
\leq Ce^{C\Gamma n^2}e^{-\frac{\Gamma^2n^2}{4}}
\esp\|y(T_n)\|_{L^2(\dom)}^2,
\end{equation}
and iterating \eqref{eq:Tn}, we deduce that for all $n\geq n_T+1$,
\begin{equation}\label{eq:Tn_prod}
\esp\|y(T_n)\|_{L^2(\dom)}^2
\leq
\prod_{k=n_T}^{n-1}
\left(Ce^{C\Gamma k^2}e^{-\frac{\Gamma^2k^2}{4}}\right)
\esp\|y_0\|_{L^2(\dom)}^2.
\end{equation}
We now fix $\Gamma>0$ {sufficiently large, satisfying
the condition in Step 1 and}
\begin{equation}\label{eq:def-Gam}
Ce^{C\Gamma n^2}\leq e^{\frac{\Gamma^2}{16}n^2},
\qquad \forall n\in\mathbb N^*.
\end{equation}
Combining \eqref{eq:Tn_prod} and \eqref{eq:def-Gam}, we infer that
\begin{equation}\label{eq:fn-es}
\esp\|y(T_n)\|_{L^2(\dom)}^2
\leq
\left(\prod_{k=n_T}^{n-1}e^{-\frac{3\Gamma^2}{16}k^2}\right)
\esp\|y_0\|_{L^2(\dom)}^2,
\qquad \forall n\geq n_T+1.
\end{equation}
In particular,
\begin{equation*}
\lim_{n\to\infty}\esp\|y(T_n)\|_{L^2(\dom)}^2=0.
\end{equation*}

Let $t\in[0,T)$ and choose $n\geq n_T$ such that $t\in I_n$.
Since $t\geq T_n$, estimate \eqref{eq:yt1} yields
\begin{equation*}
\esp\|y(t)\|_{L^2(\dom)}^2
\leq Ce^{C\Gamma n^2}\esp\|y(T_n)\|_{L^2(\dom)}^2.
\end{equation*}
Combining this with \eqref{eq:fn-es} and using again
\eqref{eq:def-Gam}, {we obtain, for $n\geq n_T+1$,
\begin{equation*}
\esp\|y(t)\|_{L^2(\dom)}^2
\leq
\exp\left(
\frac{\Gamma^2}{16}n^2
-\frac{3\Gamma^2}{16}\sum_{k=n_T}^{n-1}k^2
\right)
\esp\|y_0\|_{L^2(\dom)}^2.
\end{equation*}
Since $\sum_{k=n_T}^{n-1}k^2\sim n^3/3$ as $n\to\infty$,
}we deduce that
\begin{equation}\label{eq:yt_to_zero}
\lim_{t\to T^-}\esp\|y(t)\|_{L^2(\dom)}^2=0.
\end{equation}

Since $y$ is adapted and each operator $\mathscr H_{\lambda_n}$
is deterministic and bounded on $L^2(\dom)$, the control $h$
defined by \eqref{eq:def_feedback_global} is adapted by construction.
{
Moreover, the continuity of $y$ on each interval $[T_n,T_{n+1}]$
implies that $h$ is progressively measurable on $[0,T)$.
}
For each $n\geq n_T$, system \eqref{eq:forward_semilinear_feed_n}
is a closed-loop system with parameter $\lambda_n$.
Hence, by \Cref{theo:main}, for all $t\in I_n$,
\begin{equation*}
\esp\|h(t)\|_{L^2(\dom_0)}^2
\leq Ce^{C\sqrt{\lambda_n}}e^{-\lambda_n(t-T_n)}
\esp\|y(T_n)\|_{L^2(\dom)}^2.
\end{equation*}
Since $\lambda_n=\Gamma^2n^4$, we have
$\sqrt{\lambda_n}=\Gamma n^2$, and therefore
\begin{equation*}
\esp\|h(t)\|_{L^2(\dom_0)}^2
\leq Ce^{C\Gamma n^2}e^{-\Gamma^2n^4(t-T_n)}
\esp\|y(T_n)\|_{L^2(\dom)}^2,
\qquad t\in I_n.
\end{equation*}
Integrating over $I_n$, we obtain
\begin{align*}
\int_{I_n}\esp\|h(t)\|_{L^2(\dom_0)}^2\,\d t
&\leq Ce^{C\Gamma n^2}\esp\|y(T_n)\|_{L^2(\dom)}^2
\int_{T_n}^{T_{n+1}}e^{-\Gamma^2n^4(t-T_n)}\,\d t\\
&\leq \frac{C}{\Gamma^2n^4}e^{C\Gamma n^2}
\esp\|y(T_n)\|_{L^2(\dom)}^2.
\end{align*}
{
Combining this with \eqref{eq:fn-es}, we obtain
\begin{equation*}
\int_{I_n}\esp\|h(t)\|_{L^2(\dom_0)}^2\,\d t
\leq \frac{C}{\Gamma^2n^4}
\exp\left(
C\Gamma n^2-\frac{3\Gamma^2}{16}
\sum_{k=n_T}^{n-1}k^2
\right)
\esp\|y_0\|_{L^2(\dom)}^2,
\qquad n\geq n_T.
\end{equation*}
Since $\sum_{k=n_T}^{n-1}k^2\sim n^3/3$ as $n\to\infty$,
Tonelli's theorem gives
\begin{align*}
\esp\int_0^T\|h(t)\|_{L^2(\dom_0)}^2\,\d t
&=\sum_{n=n_T}^{\infty}
\int_{I_n}\esp\|h(t)\|_{L^2(\dom_0)}^2\,\d t\\
&\leq \frac{C}{\Gamma^2}\esp\|y_0\|_{L^2(\dom)}^2
\sum_{n=n_T}^{\infty}\frac{1}{n^4}
\exp\left(
C\Gamma n^2-\frac{3\Gamma^2}{16}
\sum_{k=n_T}^{n-1}k^2
\right)<\infty.
\end{align*}
Thus, $h\in L^2_{\fil}(0,T;L^2(\dom_0))$.
}

{
\smallskip
\textit{-- Step 3: Extension to the final time and null-controllability.}
Let $h$ be the control defined in \eqref{eq:def_feedback_global}.
As we have seen in Step 2, $h\in L^2_{\fil}(0,T;L^2(\dom_0))$.
Thus, \Cref{thm:general_wellposed}, applied with $F=0$ and
$f=\chi_{\dom_0}h$, gives a unique weak solution
\begin{equation*}
\widetilde y\in
L^2_{\fil}(\Omega;C([0,T];L^2(\dom)))
\cap L^2_{\fil}(0,T;H_0^1(\dom))
\end{equation*}
to \eqref{eq:stoch-heat-control}, with initial datum $y_0$.
We will show that $\widetilde y$ coincides on $[0,T)$ with the process
$y$ constructed in Step 1.

To this end, fix $n\geq n_T+1$ and $\phi\in H_0^1(\dom)$.
For $k=n_T,\ldots,n-1$, the construction in Step 1 gives,
almost surely for all $r\in[T_k,T_{k+1}]$,
\begin{align}\label{eq:weak_interval_nc}
(y(r)-y(T_k),\phi)_{L^2(\dom)}
&=-\int_{T_k}^{r}(\nabla y(s),\nabla\phi)_{L^2(\dom)}\,\d s
\notag\\
&\quad+\int_{T_k}^{r}
(cy(s)+\chi_{\dom_0}h(s),\phi)_{L^2(\dom)}\,\d s
\notag\\
&\quad+\sum_{\ell\geq1}\sqrt{q_\ell}
\int_{T_k}^{r}(G_\ell y(s),\phi)_{L^2(\dom)}\,\d\beta_\ell(s).
\end{align}
Let $t\in[T_j,T_{j+1}]$, where $n_T\leq j\leq n-1$.
Summing \eqref{eq:weak_interval_nc} with $r=T_{k+1}$
for $k=n_T,\ldots,j-1$, we can deduce
\begin{align}\label{eq:weak_glued_nc}
(y(t)-y_0,\phi)_{L^2(\dom)}
&=\sum_{k=n_T}^{j-1}
(y(T_{k+1})-y(T_k),\phi)_{L^2(\dom)}
+(y(t)-y(T_j),\phi)_{L^2(\dom)}
\notag\\
&=-\int_0^t(\nabla y(s),\nabla\phi)_{L^2(\dom)}\,\d s
+\int_0^t(cy(s)+\chi_{\dom_0}h(s),\phi)_{L^2(\dom)}\,\d s
\notag\\
&\quad+\sum_{\ell\geq1}\sqrt{q_\ell}
\int_0^t(G_\ell y(s),\phi)_{L^2(\dom)}\,\d\beta_\ell(s).
\end{align}
By construction, the initial datum on $[T_k,T_{k+1}]$
is the value at $T_k$ of the solution on $[T_{k-1},T_k]$.
Thus, $y$ is continuous at each point $T_k$. Since $[0,T_n]$
is the union of finitely many intervals $[T_k,T_{k+1}]$, we have
\begin{equation*}
y\in L^2_{\fil}(\Omega;C([0,T_n];L^2(\dom)))
\cap L^2_{\fil}(0,T_n;H_0^1(\dom)).
\end{equation*}

Subtracting \eqref{eq:weak_glued_nc} from the weak formulation
satisfied by $\widetilde y$, we find that
$w:=\widetilde y-y$ satisfies
\begin{align*}
(w(t),\phi)_{L^2(\dom)}
&=-\int_0^t(\nabla w(s),\nabla\phi)_{L^2(\dom)}\,\d s
+c\int_0^t(w(s),\phi)_{L^2(\dom)}\,\d s\\
&\quad+\sum_{\ell\geq1}\sqrt{q_\ell}
\int_0^t(G_\ell w(s),\phi)_{L^2(\dom)}\,\d\beta_\ell(s),
\end{align*}
almost surely for all $t\in[0,T_n]$.
Thus, $w$ solves the equation with zero initial datum and
zero forcing. By \eqref{eq:wellposed_estimate},
\begin{equation*}
\esp\left(\sup_{t\in[0,T_n]}\|w(t)\|_{L^2(\dom)}^2\right)=0,
\end{equation*}
and therefore $\widetilde y=y$ on $[0,T_n]$ almost surely.

Finally, since $T_n\uparrow T$ and $\widetilde y$ has continuous
trajectories in $L^2(\dom)$,
\begin{equation*}
\|\widetilde y(T)\|_{L^2(\dom)}^2
=\lim_{n\to\infty}\|\widetilde y(T_n)\|_{L^2(\dom)}^2
=\lim_{n\to\infty}\|y(T_n)\|_{L^2(\dom)}^2
\qquad\text{a.s.}
\end{equation*}
Fatou's lemma and \eqref{eq:yt_to_zero} now give
\begin{equation*}
\esp\|\widetilde y(T)\|_{L^2(\dom)}^2
\leq\liminf_{n\to\infty}\esp\|y(T_n)\|_{L^2(\dom)}^2=0.
\end{equation*}
Therefore, $\widetilde y(T)=0$ in $L^2(\dom)$ almost surely.
}
This completes the proof.
\end{proof}

\begin{rmk}\label{rem:time_T_gen}
The assumption $1/T\in\N^*$ is made only for simplicity. In the general case, one can choose $n\in\N^*$ such that $1/n<T$ and apply the previous construction on the interval $[0,1/n]$ to drive the solution to zero at time $1/n$. Since the system is linear, taking $h\equiv 0$ on $(1/n,T)$, the corresponding solution remains identically zero on the remaining time interval. Therefore, the result holds for any $T>0$.
\end{rmk}

\appendix

\section{Mean-square and almost sure stability}\label{app_stab}

In this appendix, we collect some classical stability properties of \eqref{eq:spde}. 
We first recall an explicit representation of the solution to \eqref{eq:spde}. Let
\begin{equation*}
\rho(t):=\exp\left(\left(c-\frac{a^2}{2}\right)t+aW(t)\right),
\qquad t\geq 0.
\end{equation*}
By It\^o's formula, we have
\begin{equation*}
\d\rho(t)=c\rho(t)\dt+a\rho(t)\d W(t),
\qquad \rho(0)=1.
\end{equation*}
Let $u(t):=e^{t\Delta}y_0$, where $e^{t\Delta}$ denotes the Dirichlet heat semigroup. For almost every $\omega\in\Omega$, the function $u$ solves
\begin{equation}\label{eq:u}
\begin{cases}
u_t=\Delta u & \textnormal{in } (0,T)\times\dom,\\
u=0 & \textnormal{on } (0,T)\times\Gamma,\\
u(0)=y_0 & \textnormal{in } \dom.
\end{cases}
\end{equation}
By It\^o's formula,
\begin{equation*}
\d(\rho^{-1})
=(a^2-c)\rho^{-1}\dt-a\rho^{-1}\d W(t).
\end{equation*}
Therefore, using \eqref{eq:spde}, we obtain
\begin{align*}
\d(\rho^{-1}y)
&=\rho^{-1}\d y+y\,\d(\rho^{-1})
-a^2\rho^{-1}y\dt\\
&=\rho^{-1}(\Delta y+cy)\dt
+a\rho^{-1}y\d W(t)\\
&\quad +(a^2-c)\rho^{-1}y\dt
-a\rho^{-1}y\d W(t)
-a^2\rho^{-1}y\dt\\
&=\rho^{-1}\Delta y\dt
=\Delta(\rho^{-1}y)\dt,
\end{align*}
where the last equality follows because $\rho$ depends only on time.
Note that since $\rho(0)=1$, the process $\rho^{-1}y$ satisfies the same
heat equation and initial condition as $u$ (in the weak sense). By uniqueness, $\rho^{-1}y=u$, and therefore
\begin{equation}\label{eq:scalar_representation}
y(t)=\rho(t)e^{t\Delta}y_0
\quad \textnormal{in } L^2(\dom),
\qquad \textnormal{a.s. for all } t\geq 0.
\end{equation}

\subsection{Mean-square exponential stability}

We derive the classical mean-square stability result for the solution to \eqref{eq:spde}. More precisely, we show that for
\begin{equation}\label{mu_choice}
\mu := 2(\tau_1 - c) - a^2,
\end{equation}
the solution satisfies
\begin{equation}\label{eq:ms_target}
\mathbb{E}\big(\|y(t)\|^2_{L^2(\dom)}\big)
\leq e^{-\mu t}\mathbb{E}\big(\|y_0\|^2_{L^2(\dom)}\big),
\quad \text{for all } t\geq 0.
\end{equation}

Using \eqref{eq:scalar_representation} and the independence of $y_0$ and $W(t)$, we obtain
\begin{align*}
\mathbb E\big(\|y(t)\|^2_{L^2(\dom)}\big)
&=\mathbb E\left(\rho(t)^2
\|e^{t\Delta}y_0\|^2_{L^2(\dom)}\right)\\
&=\mathbb E\big(\rho(t)^2\big)
\mathbb E\left(\|e^{t\Delta}y_0\|^2_{L^2(\dom)}\right).
\end{align*}
Since $W(t)$ is a centered Gaussian random variable with variance $t$, we have
\begin{equation}\label{eq:decay_rho}
\mathbb E\big(\rho(t)^2\big)
=e^{(2c-a^2)t}\mathbb E\big(e^{2aW(t)}\big)
=e^{(2c+a^2)t}.
\end{equation}
On the other hand, by estimate \eqref{eq:energy} applied to $u$ solution to \eqref{eq:u} (i.e. with $c_0=0$), yields
\begin{equation}\label{eq:decay_semi}
\|e^{t\Delta}y_0\|^2_{L^2(\dom)}
\leq e^{-2\tau_1t}\|y_0\|^2_{L^2(\dom)},
\qquad \textnormal{a.s. for all } t\geq 0.
\end{equation}
Combining \eqref{eq:decay_rho} and \eqref{eq:decay_semi}, we deduce
\begin{align*}
\mathbb E\big(\|y(t)\|^2_{L^2(\dom)}\big)
&=e^{(2c+a^2)t}
\mathbb E\left(\|e^{t\Delta}y_0\|^2_{L^2(\dom)}\right)\leq e^{-(2\tau_1-2c-a^2)t}
\mathbb E\big(\|y_0\|^2_{L^2(\dom)}\big).
\end{align*}
With the choice of $\mu$ in \eqref{mu_choice}, this proves \eqref{eq:ms_target}. In particular, if $2\tau_1>a^2+2c$, then $\mu>0$ and the solution is mean-square exponentially stable.

\subsection{Almost sure exponential stability}

We now derive the almost sure exponential stability estimate. More precisely, we show that
\begin{equation}\label{eq:as_target}
\limsup_{t\to\infty}\frac{1}{t}\log \|y(t)\|^2_{L^2(\dom)} \leq -\mu
\quad \text{a.s.},
\end{equation}
where $\mu:=2(\tau_1-c)+a^2$.

Using \eqref{eq:scalar_representation} and \eqref{eq:decay_semi}, we obtain
\begin{align*}
\|y(t)\|^2_{L^2(\dom)}
&=\exp\left((2c-a^2)t+2aW(t)\right)
\|e^{t\Delta}y_0\|^2_{L^2(\dom)}\\
&\leq \exp\left(-\left(2(\tau_1-c)+a^2\right)t+2aW(t)\right)
\|y_0\|^2_{L^2(\dom)},
\end{align*}
almost surely and for all $t\geq 0$.

On the event $\{y_0=0\}$, the solution is identically zero, so \eqref{eq:as_target} holds with the convention $\log 0=-\infty$. On the event $\{y_0\neq 0\}$, taking logarithms and dividing by $t>0$ gives
\begin{equation}\label{eq:estimate_log}
\frac{1}{t}\log\|y(t)\|^2_{L^2(\dom)}
\leq -2(\tau_1-c)-a^2
+\frac{2aW(t)}{t}
+\frac{1}{t}\log\|y_0\|^2_{L^2(\dom)}.
\end{equation}
By the law of large numbers,
\begin{equation*}
\lim_{t\to+\infty}\frac{W(t)}{t}=0
\qquad \textnormal{a.s.}
\end{equation*}
Moreover, since $y_0\in L^2_{\mathcal F_0}(\Omega;L^2(\dom))$, $\|y_0\|^2_{L^2(\dom)}<+\infty$ almost surely. Hence, on $\{y_0\neq 0\}$,
\begin{equation*}
\lim_{t\to+\infty}\frac{1}{t}
\log\|y_0\|^2_{L^2(\dom)}=0
\qquad \textnormal{a.s.}
\end{equation*}
Passing to the limit in \eqref{eq:estimate_log}, we conclude that
\begin{equation*}
\limsup_{t\to+\infty}\frac{1}{t}
\log\|y(t)\|^2_{L^2(\dom)}
\leq -2(\tau_1-c)-a^2
\qquad \textnormal{a.s.}
\end{equation*}
This proves \eqref{eq:as_target}. In particular, if $2\tau_1+a^2>2c$, then $\mu>0$ and the solution is almost surely exponentially stable.

\bibliographystyle{alpha}
\small{\bibliography{bib_stab}}

\bigskip

\begin{flushleft}

\textbf{Víctor Hernández-Santamaría} and \textbf{Liliana Peralta}\\
Departamento de Matem\'aticas, Facultad de Ciencias\\
Universidad Nacional Autónoma de México \\
Circuito Exterior, C.U.\\
04510, Coyoacán, CDMX, Mexico\\
\texttt{victor.santamaria@ciencias.unam.mx \\ lylyaanaa@ciencias.unam.mx}

\bigskip
\bigskip

\textbf{Kévin Le Balc'h}\\
Laboratoire Jacques-Louis Lions \\
Inria, Sorbonne Université\\
Université de Paris, CNRS \\
Paris, France\\
\texttt{kevin.le-balc-h@inria.fr}

\end{flushleft}

\end{document}